\documentclass[10pt]{amsart}

\usepackage[latin2]{inputenc}
\usepackage{amsmath}
\usepackage{graphicx}
\usepackage{amssymb}
\usepackage{esint}
\usepackage{color}
\usepackage{floatrow}
\usepackage{float,verbatim}
\usepackage{amsthm}
\usepackage{tikz}
\usetikzlibrary{calc}
\usepackage{epsfig}
\usepackage[english]{babel}

\newtheorem{theorem}{Theorem}
\newtheorem{proposition}[theorem]{Proposition}
\newtheorem{lemma}[theorem]{Lemma}

\newtheorem{definition}[theorem]{Definition}
\newtheorem{remark}[theorem]{Remark}

\newtheorem*{theorem*}{Theorem}

\def\XXint#1#2#3{{\setbox0=\hbox{$#1{#2#3}{\int}$ }
\vcenter{\hbox{$#2#3$ }}\kern-.6\wd0}}

\newcommand{\adj}{\operatorname{adj}}
\newcommand{\Id}{\operatorname{Id}}
\newcommand{\id}{\operatorname{id}}

\definecolor{Yellow}{rgb}{0.95,0.9,0.0}
\definecolor{Red}{rgb}{0.8,0.1,0.1}
\definecolor{Green}{rgb}{0.1,0.65,0.2}
\definecolor{Blue}{rgb}{0.1,0.1,0.8}
\definecolor{Purple}{rgb}{0.7,0.1,0.7}
\definecolor{Grey}{rgb}{0.6,0.6,0.6}

\begin{document}

\title[The Prescribed Jacobian Inequality with $L^p$ Data]
{Bi-Sobolev Solutions to the Prescribed Jacobian Inequality in the Plane with $L^p$ Data}

\author{Julian Fischer and Olivier Kneuss}
\address{Max Planck Institute for Mathematics in the Sciences,
Inselstrasse 22, 04103 Leipzig, Germany, E-Mail:
julian.fischer@mis.mpg.de}
\address{Institute of Mathematics, Federal University of Rio de Janeiro, Cidade Universitaria, 21941909 Rio de Janeiro, Brazil, E-Mail:
olivier.kneuss@gmail.com}
\begin{abstract}
We construct planar bi-Sobolev mappings whose local volume distortion is bounded from below by a given function $f\in L^p$ with $p>1$, i.\,e.\ bi-Sobolev solutions for the prescribed Jacobian inequality in the plane for right-hand sides $f\in L^p$. More precisely, for any $1<q<(p+1)/2$ we construct a solution which belongs to $W^{1,q}$ and which preserves the boundary pointwise. For bounded right-hand sides $f\in L^{\infty}$, we provide bi-Lipschitz solutions. The basic building block of our construction are Lipschitz maps which stretch a given compact subset of the unit square by a given factor while preserving the boundary. The construction of these stretching maps relies on a slight strengthening of the covering result of Alberti, Cs\"ornyei, and Preiss for measurable planar sets in the case of compact sets.
\end{abstract}
\maketitle
\section{Introduction}
The prescribed Jacobian equation
\begin{subequations}
\label{JacobianEqualityWithBC}
\begin{align}
\label{JacobianEquality}
\det\nabla \phi&=f &&\text{in }\Omega,
\\
\phi&=\operatorname{id} &&\text{on }\partial \Omega,
\end{align}
\end{subequations}
(here for bounded connected domains $\Omega\subset \mathbb{R}^d$, positive right-hand sides $f:\Omega\rightarrow (0,\infty)$, and maps $\phi:\overline{\Omega}\rightarrow \mathbb{R}^d$) is an important subject of studies in geometric analysis. By the change of variables formula, the equation amounts to prescribing the volume distortion of the mapping $\phi$.

Clearly, the prescribed Jacobian equation \eqref{JacobianEquality} is underdetermined, as it is invariant under composition of $\phi$ with an arbitrary volume-preserving map. Therefore, one has neither uniqueness of solutions nor a regularity theory for solutions: Even for smooth right-hand sides $f$, there exist infinitely many solutions with rather low regularity. On the other hand, the regularity of the right-hand side $f$ may obviously restrict the regularity of solutions $\phi$: For right-hand sides $f$ that fail to belong to the H\"older space $C^{k,\alpha}$, the solutions $\phi$ cannot belong to the H\"older space $C^{k+1,\alpha}$. Similarly, for right-hand sides $f$ that lack $L^p$-integrability, solutions $\phi$ cannot belong to the Sobolev space $W^{1,dp}$.

Beyond these trivial restrictions on the regularity of solutions $\phi$, a notable result of Burago and Kleiner \cite{Burago-Kleiner} and McMullen \cite{McMullen} establishes the existence of bounded right-hand sides $f\in L^\infty$ for which the prescribed Jacobian equation \eqref{JacobianEquality} does not admit bi-Lipschitz solutions. In fact, these right-hand sides $f$ may even be taken as continuous and arbitrarily close to $1$.

For H\"older-regular right-hand sides $f\in C^{k,\alpha}(\overline{\Omega})$ (with $k\in \mathbb{N}_0$, $0<\alpha<1$) and sufficiently regular connected domains $\Omega$, an existence theory with optimal regularity has been developed for the prescribed Jacobian equation with identity boundary conditions \eqref{JacobianEqualityWithBC}: Provided that the trivial necessary condition $\int_{\Omega}f\,dx=|\Omega|$ for the solvability of \eqref{JacobianEqualityWithBC} is satisfied, the existence of a solution $\phi\in C^{k+1,\alpha}(\overline{\Omega};\overline{\Omega})$ has been shown by Dacorogna and Moser \cite{Dac-Moser}. Later, alternative proofs were given by Rivi\`ere and Ye \cite{Riv-Ye}, Avinyo et.\ al.\ \cite{AvinyoEtAl}, and Carlier and Dacorogna \cite{Dac-Carlier}. Note also that Ye \cite{Ye} proved that for right-hand sides $f$ satisfying the above necessary condition and  belonging to $W^{m,p}(\Omega),$ where $m\in \mathbb{N}$ and  $1<p<\infty$ are such that $W^{m,p}(\Omega)$ is an algebra, there exists a solution to \eqref{JacobianEqualityWithBC} in $W^{m+1,p}(\Omega;\Omega).$ An overview of the available theory for the prescribed Jacobian equation may be found in \cite{CDK}.

While the existence theory for the prescribed Jacobian equation is well-developed in the case of H\"older spaces, in the regime of less regular right-hand sides the situation changes drastically. In the case of merely $L^p$-integrable right-hand sides $f$, it is an open question for all $p\leq \infty$ whether in general a weak solution of the prescribed Jacobian equation \eqref{JacobianEqualityWithBC} exists in some Sobolev space $W^{1,q}(\Omega;\mathbb{R}^d)$ for some $q\geq 1$. The only existence result in this regime is due to  Rivi\`ere and Ye \cite{Riv-Ye} and ensures the existence of H\"older-continuous solutions $\phi\in C^{0,\alpha}$ for right-hand sides $f\in L^\infty$ or $f\in \operatorname{BMO}$. Here, ``solution'' is to be understood in the distributional sense, as defined by the change of variables formula.

In  the present work, in the planar case (i.\,e.\ for $\Omega\subset \mathbb{R}^2$) we establish existence results in Sobolev spaces $W^{1,q}(\Omega;\mathbb{R}^d)$ for the prescribed Jacobian \emph{inequality}
\begin{subequations}
\label{JacobianInequalityBC}
\begin{align}
\label{JacobianInequality}
\det\nabla \phi&\geq f\quad &&\text{in }\Omega,
\\
\phi&=\operatorname{id}\quad &&\text{on }\partial \Omega,
\end{align}
\end{subequations}
assuming just appropriate integrability of the right-hand side $f$. More precisely, for nonnegative right-hand sides $f\geq 0$ satisfying $\int_\Omega f~dx<|\Omega|$, we construct solutions with the following properties:
\begin{itemize}
\item Given $f\in L^p(\Omega)$ for some $p>1$, we show in Theorem \ref{L.p} that for every $1<q<(p+1)/2$ there exists a homeomorphism $\phi:\overline{\Omega}\rightarrow \overline{\Omega}$ which solves the prescribed Jacobian inequality \eqref{JacobianInequalityBC} and for which both $\phi$ and its inverse $\phi^{-1}$ belong to the Sobolev space $W^{1,q}(\Omega;\Omega)$.
\item For bounded right-hand sides $f\in L^{\infty}(\Omega)$, the prescribed Jacobian inequality \eqref{JacobianInequalityBC} admits a bi-Lipschitz solution $\phi:\overline{\Omega}\rightarrow \overline{\Omega}$, see Theorem \ref{theorem.det.nabla.u.geq.f}. The bi-Lipschitz regularity is in general sharp. Recall that this existence result is in strong contrast to the case of the prescribed Jacobian equation \eqref{JacobianEquality}, for which the existence of bi-Lipschitz solutions for bounded right-hand sides fails in general, as mentioned above \cite{Burago-Kleiner,McMullen}.
\end{itemize}
Note that the condition on the right-hand side $\int_{\Omega}f~dx < |\Omega|$ is indeed a necessary condition for the problem of the prescribed Jacobian inequality to be meaningful: In the case $\int_{\Omega}f~dx > |\Omega|$, the change of variables formula excludes the existence of (weak) solutions; in the equality case $\int_{\Omega}f~dx = |\Omega|$, by the change of variables formula the prescribed Jacobian inequality \eqref{JacobianInequalityBC} reduces to the more ``rigid'' case of the prescribed Jacobian equation \eqref{JacobianEquality}.

To the best of our knowledge, Theorem \ref{L.p} is the first existence result giving a solution (of an equation involving the Jacobian) whose gradient has some integrability when the right-hand side $f$ is merely in $L^p(\Omega)$.

For maps $\phi\in W^{1,q}(\Omega;\mathbb{R}^d)$ with $q\geq d$, the Jacobian determinant $\det\nabla \phi$ is naturally defined in the pointwise sense. In the case $d-1<q<d$, in the present work the Jacobian will be understood in the distributional sense (see e.\,g.\ \cite{DePhilippis,DeLellis,dacorogna}), namely
\begin{align*}
\mathcal{J}_{\phi}[\eta]:=
-\frac{1}{d}\int_\Omega \langle\adj \nabla \phi \cdot \phi, \nabla \eta\rangle ~dx\quad \text{for every test function }\eta \in C^\infty_{cpt}(\Omega).
\end{align*}
In the supercritical case $q>d$ and the critical case $q=d$, the pointwise notion of the Jacobian determinant coincides with the distributional determinant, i.\,e.\ it holds that
\begin{align*}
\mathcal{J}_\phi[\eta]=\int_\Omega \eta \det \nabla \phi ~dx.
\end{align*}
In the subcritical case $q<d$, the two notions differ in general: A simple example is provided by the map
\begin{align*}
\phi(x)=\frac{x}{|x|}.
\end{align*}
In fact, in the subcritical case the Jacobian determinant defined in the \emph{pointwise} sense experiences a certain loss of rigidity and is no longer associated with the change of variables formula. This has in particular been exploited by Koumatos, Rindler, and Wiedemann \cite{KKR} to construct solutions to the prescribed Jacobian equation in the pointwise sense with right-hand sides $f\in L^p(\Omega)$, $p<d$, $\Omega\subset \mathbb{R}^d$, by employing convex integration techniques.



Before stating our main results, let us comment on the strategy for the proof of our results. We shall reduce both existence results for the prescribed Jacobian inequality (\ref{JacobianInequality}) Theorem~\ref{L.p} and Theorem~\ref{theorem.det.nabla.u.geq.f} (i.\,e.\ both the case of right-hand sides in $L^p$ and the case of right-hand sides in $L^\infty$) to the existence of bi-Lipschitz maps which stretch a given measurable subset of the plane. This stretching result is provided in Theorem \ref{BiLipschitzMapTheorem} and is at the heart of our paper. In the case of right-hand sides $f\in L^p$, the reduction to Theorem \ref{BiLipschitzMapTheorem} is facilitated by iterative stretching of appropriate superlevel sets of $f$; see the proof of Proposition \ref{proposition.assuming.f.L.p.small}. In the case of bounded right-hand sides $f\in L^\infty$, the reduction to Theorem \ref{BiLipschitzMapTheorem} is quite straightforward and relies on the classical existence results for the prescribed Jacobian equation in the smooth case.

In Theorem \ref{BiLipschitzMapTheorem}, we consider a bounded open set $\Omega$ in $\mathbb{R}^2$ with $C^{1,1}$ boundary, let $\tau>0$, and consider any measurable set $M\subset \Omega$ with small measure. Then we are able to construct a bi-Lipschitz mapping $\phi=\phi_{\tau,M}$ (with $\phi:\overline{\Omega}\rightarrow \overline{\Omega}$) preserving the boundary pointwise, stretching $M$ by a factor of at least $1+\tau$, and compressing in a controlled way outside $M$ in the sense $\det\nabla \phi\geq 1+\tau$ a.\,e.\ on $M$ and $\det\nabla\phi\geq 1-C|M|^{1/2}\tau$ a.\,e.\ in $\Omega$,
together with uniform $W^{1,p}$ estimates for $1\leq p\leq \infty$.
Our proof of Theorem \ref{BiLipschitzMapTheorem} is constructive and is based on a slight strengthening of a covering result for planar measurable sets by Alberti, Cs\"ornyei, and Preiss \cite{Alberti}. As the covering result is restricted to two dimensions, the same is true for our construction (see Remark \ref{remark.no.generalization.higher.dim}).

The covering result of Alberti, Cs\"ornyei, and Preiss \cite{Alberti} states that any compact subset $K$ of the unit square may be covered (for $\delta>0$ small enough) by a collection of horizontal and vertical 1-Lipschitz strips with height (respectively width) $2\delta$, the number of strips being bounded by $2\delta^{-1}\sqrt{|K|}$.
Here, a horizontal
(respectively vertical) 1-Lipschitz strip is defined to be the graph of a
1-Lipschitz function over the horizontal (respectively vertical) axis thickened
by $\delta$ in the vertical (respectively horizontal) direction.
In fact, the strips may be chosen in such a way that the horizontal strips have pairwise disjoint interior and such that the same is true for the vertical strips (see Lemma \ref{IntersectionFreeCovering}). This has first been observed by Marchese \cite{Marchese}; the authors of the present paper have established a similar result independently in an earlier version of the present work. See the picture in the center of Figure~\ref{SketchStretching} for an example of such a covering by strips.

By stretching the horizontal strips in the vertical direction and the vertical strips in the horizontal direction (see the bottom pictures in Figure~\ref{SketchStretching}), in Proposition \ref{Prop.Bi.Lip.simpl} we obtain a mapping which stretches a given compact subset $K$ of the unit square. Here, the fact that the strips may be chosen in such a way that the number of strips that may cover a single point is bounded seems to be crucial for our construction, as otherwise the uniform Lipschitz bound for our maps would be lost.
Next, by an explicit construction based on an interpolation on a boundary layer between the identity boundary conditions and the boundary values provided by Proposition \ref{Prop.Bi.Lip.simpl}, we obtain a map that additionally preserves the boundary pointwise (see Lemma \ref{lemma.simpl.boundary.values} and, in particular, Figure~\ref{BoundaryCorrection}). Using inner regularity of the Lebesgue measure and weak continuity of the determinant, our result immediately extends to general measurable sets (see Lemma \ref{Lemma.simpl.set}). Finally we extend the result -- which by now has only been proven for the unit square -- to domains in the plane with boundary of class $C^{1,1}$; see Lemma \ref{Lemma.simpl.domain}.

\smallskip

{\bf Notation.}
Throughout the paper, we use the following notation:
\begin{itemize}
\item We denote the Lebesgue measure of a measurable set $M\subset
\mathbb{R}^d$ by $|M|.$
\item For $M\subset \mathbb{R}^d,$ $\chi_{M}$ denotes the usual characteristic
function of $M:$
\begin{align*}
\chi_M=\left\{\begin{array}{cl}
1&\text{in $M$}\\
0&\text{in $M^c.$}
\end{array}
\right.
\end{align*}
\item The notation $\sharp P$ refers to the number of elements
of the set $P$.
\item A function $g:\mathbb{R}\rightarrow \mathbb{R}$ is said to be
$1$-Lipschitz if it is Lipschitz with $|g'|\leq 1$ almost everywhere.
\item By $C$ we denote a generic constant whose value may change from
appearance to appearance.
\item By $\operatorname{id}$ we denote the identity map, while by
$\operatorname{Id}$ we denote the unit matrix.
\item For a matrix $A\in \mathbb{R}^2$ we denote by $\operatorname{adj}A$ the adjugate matrix of $A.$
For example when $d=2$ we have
\begin{align*}
\operatorname{adj}A=\begin{pmatrix}
A^2_2
&
-A^1_2
\\
-A^2_1
&
A^1_1
\end{pmatrix}\quad \text{where}\quad
A=\begin{pmatrix}
A^1_1
&
A^1_2
\\
A^2_1
&
A^2_2
\end{pmatrix}.
\end{align*}
\item Throughout the paper, we use standard notation for Sobolev and H\"older spaces. For example, by $C^{1,\alpha}(\Omega)$ we denote the space of functions on $\Omega$ whose first derivative is H\"older continuous with exponent $\alpha$; by $W^{1,p}(\Omega;\mathbb{R}^d)$, we denote the space of functions $f\in L^p(\Omega)$ whose distributional derivative satisfies $\nabla f\in L^p(\Omega;\mathbb{R}^d)$. By $\operatorname{Diff}^\infty(A;B)$ we denote the class of smooth diffeomorphisms from $A$ to $B$.
\end{itemize}

\section{Main results}
We now state our main results.

\begin{theorem}\label{L.p}
Let $\Omega\subset\mathbb{R}^2$ be a bounded connected domain with boundary of class $C^{1,1}$. Let $f\in L^p(\Omega)$, $p>1$, be nonnegative with
\begin{align*}
\int_\Omega f \,dx<|\Omega|.
\end{align*}
Then for every $q$ with $1<q<(p+1)/2$ there exists a homeomorphism $\phi:\overline{\Omega}\rightarrow \overline{\Omega}$ with the bi-Sobolev property $\phi,\phi^{-1}\in W^{1,q}(\Omega;\Omega)$ which solves the prescribed Jacobian inequality with identity boundary conditions
\begin{subequations}
\label{main.equation.L.p.case}
\begin{align}
\label{MainResultJacobianInequality}
\det\nabla \phi&\geq f &&\text{in }\Omega,
\\
\phi&=\operatorname{id} &&\text{on }\partial \Omega.
\end{align}
\end{subequations}
Here, the inequality \eqref{MainResultJacobianInequality} holds almost everywhere in the case $q\geq 2$; in contrast, in the case $q<2$ it is to be understood in the weak sense
\begin{align*}
\mathcal{J}_{\phi}[\eta]=
-\frac{1}{2}\int_\Omega \langle\adj \nabla \phi \cdot \phi, \nabla \eta\rangle ~dx\geq \int_{\Omega}f\eta ~dx\quad
\text{for every }\eta\in C^{\infty}_{cpt}(\Omega; [0,\infty)).
\end{align*}
\end{theorem}

Recall that by the change of variables formula, the condition $\int_{\Omega}\det\nabla \phi~dx\leq |\Omega|$ is necessary for the solvability of \eqref{main.equation.L.p.case}. This assertion is also true for $q<2$ when interpreting the prescribed Jacobian inequality \eqref{MainResultJacobianInequality} in the distributional sense. Recall also that for $\int_{\Omega}f=|\Omega|$ the prescribed Jacobian inequality \eqref{main.equation.L.p.case} actually reduces to the prescribed Jacobian equation (\ref{JacobianEquality}).

\begin{remark}\label{remark.apres.L.p}
Note that the best possible regularity for a solution of \eqref{MainResultJacobianInequality} (recall that we are working in two dimensions) would be $\phi\in W^{1,2p}$ (i.e. $q=2p$). However it seems that the asymptotic ratio of $1/2$ between $q$ and $p$ in our main result cannot be improved with our approach.
\end{remark}

For $p=+\infty$, the following limiting version of the previous theorem holds.
\begin{theorem}
\label{theorem.det.nabla.u.geq.f}
Let $\Omega\subset \mathbb{R}^2$ be a bounded domain with boundary of class $C^{1,1}$. Let $f\in L^{\infty}(\Omega)$ be a nonnegative function with $\int_{\Omega}f~dx<|\Omega|$. Then there exists a bi-Lipschitz map $\phi:\overline{\Omega}\rightarrow \overline{\Omega}$ satisfying
\begin{subequations}
\label{Jacobian.inequality.theorem}
\begin{align}
\det\nabla \phi&\geq f&&\text{a.\,e.\ in }\Omega,
\\
\phi&=\operatorname{id}&&\text{on }\partial \Omega.
\end{align}
\end{subequations}
Moreover the regularity of $\phi$ is sharp in general: There exists an $f$ for which there is no $C^1$ solution $\phi$ to (\ref{Jacobian.inequality.theorem}).
\end{theorem}
\begin{remark}\label{Remark.after.det.nabla.phi.geq.f}
Recall that for $f\in C^0(\overline{\Omega})$ with $f\geq 0$ and $\int_{\Omega}f~dx=|\Omega|$, in general the prescribed Jacobian \eqref{JacobianEquality} does not admit a bi-Lipschitz solution \cite{Burago-Kleiner,McMullen}. As in the case $\int_{\Omega}f=|\Omega|$ the prescribed Jacobian inequality (\ref{Jacobian.inequality.theorem}) reduces to the case of the prescribed Jacobian equation \eqref{JacobianEquality}, we see that the assumption $\int_\Omega f<|\Omega|$ is sharp in general.
\end{remark}


All the above results strongly rely on the following result which
essentially states that for any $\tau>0$ and any measurable set $M\subset \Omega\subset \mathbb{R}^2$ with small enough measure, we can construct a bi-Lipschitz map mapping $\overline{\Omega}$ to itself which
\begin{itemize}
\item  stretches the set $M$ by a factor of at least $1+\tau,$
\item compresses $\Omega\setminus M$ by no more than a factor of $1-C\sqrt{|M|}\tau,$
\item has good uniform estimates for its difference from the identity,
\item preserves the boundary pointwise.
\end{itemize}

\begin{theorem}
\label{BiLipschitzMapTheorem}
Let $\Omega\subset \mathbb{R}^2$ be a bounded domain with boundary of class
$C^{1,1}$. Then there exist  constants $C,c>0$ depending only on $\Omega$ with
the following property: For any $\tau>0$ and for any measurable set $M\subset
\Omega$ with $\max\{|M|,\sqrt{|M|}\tau\}\leq c$ there exists a bi-Lipschitz mapping
$\phi=\phi_{\tau,M}:\overline{\Omega}\rightarrow \overline{\Omega}$ satisfying
\begin{align}
\phi&=\operatorname{id} &&\text{on }\partial\Omega,
\label{main.boundary}
\\
\|\phi-\operatorname{id}\|_{W^{1,p}(\Omega)}&\leq C|M|^{1/(2p)}
\tau &&\text{for every }1\leq p\leq \infty,
\label{main.W.1.p}
\\
\|\phi-\operatorname{id}\|_{L^\infty(\Omega)}&\leq C|M|^{1/2}\tau,
\label{main.L.infty}
\\
|\nabla\phi|&\leq C\det \nabla \phi &&\text{a.\,e.\ in }\Omega,
\label{main.nabla.phi.bounded.by.det}
\\
\det \nabla \phi&\geq 1+\tau\quad &&\text{a.\,e.\ on }M,
\label{main.det.on.M}
\\
\det \nabla \phi&\geq 1-C|M|^{1/2}\tau &&\text{a.\,e.\ in }\Omega.
\label{main.det.global}
\end{align}
\end{theorem}
\begin{remark}\label{remark.after.Th.Bilip.Stetching}
(i) We would like to emphasize that the constant $C$ only depends on
$\Omega$. In particular (\ref{main.W.1.p}) with $p=\infty$ gives the
uniform Lipschitz estimate
\begin{align*}
\|\phi-\operatorname{id}\|_{W^{1,\infty}(\Omega)}\leq C\tau.
\end{align*}
(ii) Since, from the changes of variables formula, we always have $\int_{\Omega}\det\nabla \phi=|\Omega|$
we deduce that a smallness assumption on $|M|$
(depending on $\tau$) is needed for (\ref{main.det.on.M}) to hold true.

\end{remark}

A consequence of the results of the present work is a certain generalization of results on functionals in nonlinear elasticity, at least in the planar case. In elasticity, a natural requirement for functionals $\mathcal{F}[u]$ that associate an elastic energy to a given deformation $u:\Omega\subset \mathbb{R}^d \rightarrow \mathbb{R}^d$ is that they should penalize compression of matter to zero volume by infinite energy. A mathematical model case of such functionals is the functional
\begin{align}
\label{FunctionalElasticity}
\mathcal{F}[u]:=\int_\Omega |\nabla u|^2 + h(\det \nabla u) \,dx,
\end{align}
where $h:(0,\infty)\rightarrow [0,\infty)$ is a convex function that blows up at $0$. In particular, for power-law blowup $h(s):=s^{-p}$, one has
\begin{align}
\label{FunctionalElasticityPowerLaw}
\mathcal{F}[u]:=\int_\Omega |\nabla u|^2 + \frac{1}{(\det \nabla u)_+^p} \,dx.
\end{align}
For minimizers of such functionals, the concept of equilibrium equations has been developed by Ball \cite{Ball} as a substitute for the Euler-Lagrange equations, as it is not known whether in general the Euler-Lagrange equations are satisfied by minimizers of such functionals. In contrast to the Euler-Lagrange equations, which for a minimizer $u$ are derived by passing to the limit $\tau\searrow 0$ for some smooth compactly supported test vector field $\xi$ in the relation
\begin{align}
\label{AnsatzEulerLagrangeEquation}
\frac{\mathcal{F}[u+\tau \xi]-\mathcal{F}[u]}{\tau} \geq 0,
\end{align}
the equilibrium equations are obtained by passing to the limit $\tau \searrow 0$ in the relations
\begin{align}
\label{AnsatzEquilibriumEquations}
\frac{\mathcal{F}[u\circ (\id+\tau \xi)]-\mathcal{F}[u]}{\tau} \geq 0
\quad\quad\text{or}\quad\quad
\frac{\mathcal{F}[(\id+\tau \xi)\circ u]-\mathcal{F}[u]}{\tau} \geq 0.
\end{align}
In the ansatz for the Euler-Lagrange equation \eqref{AnsatzEulerLagrangeEquation}, one cannot exclude that the competitor deformation $u+\tau \xi$ may have infinite energy for all $\tau>0$, thereby obstructing the derivation of the Euler-Lagrange equations: In general, one cannot exclude that $\det \nabla (u+\tau\xi)$ becomes negative on a set of positive measure for every $\tau>0$.

It turns out that for functionals with power-law blowup for $\det\nabla u \searrow 0$ like \eqref{FunctionalElasticityPowerLaw}, the direct ansatz for the derivation of the equilibrium equations \eqref{AnsatzEquilibriumEquations} is successful \cite{Ball}: In particular, one directly sees that for minimizers $u$ with finite energy $\mathcal{F}[u]$, the competitor deformations $u\circ (\id+\tau \xi)$ and $(\id+\tau \xi)\circ u$ have finite energy for $\tau>0$ small enough. However, in the case of superpolynomial blowup of $h$ for $\det \nabla u \searrow 0$ -- or if blowup already occurs when $\det \nabla u$ drops below some positive threshold $\delta>0$ --~, the direct ansatz \eqref{AnsatzEquilibriumEquations} for the derivation of the equilibrium equations fails, as one again cannot exclude that the competitor maps $u\circ (\id+\tau \xi)$ and $(\id+\tau \xi)\circ u$ may have infinite energy for all $\tau>0$.

In the planar case, the stretching maps provided by Theorem \ref{BiLipschitzMapTheorem} enable us to establish the equilibrium equations also for functionals \eqref{FunctionalElasticity} with a convex differentiable function $h(s)$ which blows up superpolynomially for $s\rightarrow 0$ or even already at a positive threshold $\delta>0$. The details will be provided in a forthcoming work.

\section{The prescribed Jacobian inequality in the $L^{p}$ case}
In this section we prove the existence of solutions to the prescribed Jacobian inequality for right-hand sides of class $L^p$ (i.\,e.\ Theorem \ref{L.p}). We shall reduce Theorem~\ref{L.p} to Proposition~\ref{proposition.assuming.f.L.p.small} below, using classical existence theorems for the prescribed Jacobian equation in the smooth case.  Besides some extra estimates, the statement of Proposition~\ref{proposition.assuming.f.L.p.small} corresponds to Theorem~\ref{L.p} with the additional assumption that $\|f\|_{L^p}$ is sufficiently small.

The idea for the proof of Proposition~\ref{proposition.assuming.f.L.p.small} -- which corresponds to the main step of the proof of Theorem~\ref{L.p} -- is to inductively use Theorem \ref{BiLipschitzMapTheorem} to obtain mappings which stretch well-chosen sets (defined in terms of the superlevel sets of $f$) and to show that the (infinite) composition of these mappings converges to a solution.

\begin{proposition}\label{proposition.assuming.f.L.p.small} Let $\Omega\subset\mathbb{R}^2$ be a bounded  domain with  boundary of class $C^{1,1}$. Let $p>1$ and $1< q<(p+1)/2$.
Then there exists two constants $\widetilde{c}$ and $\widetilde{C}$ depending only on $p$, $q$ and $\Omega$ with the following property: For any $f\in L^p(\Omega)$ with
$$f\geq 0\quad \text{and}\quad \|f\|_{L^p}\leq \widetilde{c},$$ there exists
a homeomorphism $\phi:\overline{\Omega}\rightarrow \overline{\Omega}$ with
$\phi,\phi^{-1}\in W^{1,q}(\Omega;\Omega)$ satisfying
\begin{subequations}\label{equation.det.gen}
\begin{align}
\label{equat.det.simplifier}
\det\nabla \phi &\geq f && \text{in }\Omega,
\\
\label{equat.det.simplifier.boundary}
\phi &= \operatorname{id} && \text{on }\partial \Omega,
\end{align}
\end{subequations}
together with the estimates
\begin{equation}\label{est.det.simpl}
\det\nabla \phi\geq 1-\widetilde{C}\|f\|_{L^p}^{p/2} \quad\quad \text{in $\Omega$}
\end{equation}
and
\begin{equation}\label{est.psi.w.1.q}
\|\phi-\operatorname{id}\|_{W^{1,q}}\leq \widetilde{C}\|f\|_{L^p}^{p/(2q)}.
\end{equation}
Here, the inequalities \eqref{equat.det.simplifier} and \eqref{est.det.simpl} are understood in the pointwise a.\,e.\ sense for $q\geq 2$ and in the distributional sense for $q<2$.
\end{proposition}
\begin{proof}
\textit{Step 1 (stretching the superlevel sets of $f$).}
We choose $0<\lambda<1$ small enough so that
\begin{align}\label{relation.q.p}
2(1+\lambda)(q-1)<p-1.
\end{align}
Let $C$ be the constant (depending only on $\Omega$) appearing in the statement of Theorem \ref{BiLipschitzMapTheorem}.
Fix $\tau_0\geq 1$ large enough so that
\begin{align}\label{tau.big.enough}
C^{1/(q-1)}(1+C\tau_0)\leq (1+\tau_0)^{1+\lambda}.
\end{align}

Proceeding by induction, we claim that, taking $\|f\|_{L^p}\leq \widetilde{c}$  (where $\widetilde{c}$ will be chosen small enough and will only depend on $p,q$ and $\Omega$), we can construct, for every $j\geq 1$, a bi-Lipschitz map
$\phi_j:\overline{\Omega}\rightarrow \overline{\Omega}$ such that
\begin{align}
\phi_j&=\operatorname{id}\quad \text{on $\partial\Omega,$}
\\
\|\phi_j-\operatorname{id}\|_{W^{1,t}(\Omega)}&\leq C\tau_0|M_j|^{1/(2t)}\quad
\text{for every }1\leq t\leq \infty,\label{main.W.1.p.i}
\\
\|\phi_{j} -\operatorname{id}\|_{L^{\infty}}&\leq C\tau_0|M_{j}|^{1/2},\label{main.L.infty.i}\\
|\nabla\phi_j|&\leq C\det \nabla \phi_j\quad \text{a.\,e.\ in }\Omega,
\label{main.nabla.phi.bounded.by.det.i}\\
\det \nabla \phi_j&\geq 1+\tau_0\quad \text{a.\,e.\ in }M_j,\label{main.det.on.M.i}
\\
\det \nabla \phi_j&\geq 1-C\tau_0|M_j|^{1/2}\quad \text{a.\,e.\ in $\Omega,$}
\label{main.det.global.i}
\end{align}
where
\begin{align}\label{definition.M.j}
M_j:=\phi_{j-1}\circ\cdots \circ\phi_1\left(
\underbrace{
\bigcap_{0\leq k\leq j-1}
\left\{ \frac{f}{\det \nabla (\phi_{k}\circ \ldots \circ \phi_1)}\geq \frac{1}{2}\right\}}_{=:A_j}
\right)
\end{align}
with the convention that $M_1=\{ f\geq 1/2\}$ and $A_1=\{f\geq \frac{1}{2}\}$ and where $C=C(\Omega)$ is as before the constant in the statement of Theorem \ref{BiLipschitzMapTheorem}.
\smallskip

\textit{Step 1.1 (start of induction).}
First, since obviously
\begin{align*}
|\{f\geq 1/2\}|\rightarrow 0\quad \text{as $\|f\|_{L^p}\rightarrow 0$,}
\end{align*}
taking $\|f\|_{L^p}$  small enough,
we can apply Theorem \ref{BiLipschitzMapTheorem} with $\tau=\tau_0$ and
$M=\{ f\geq 1/2\}$
to get $\phi_1.$\smallskip

\textit{Step 1.2 (induction step).}
Assume that $\phi_1,\cdots,\phi_{i-1}$ have been constructed. We now show how to obtain $\phi_i$.
Again the existence of $\phi_i$ will be a direct consequence of Theorem \ref{BiLipschitzMapTheorem} once we have shown that the measure of $M_i$ is small enough so that $\max\{|M_i|,\sqrt{|M_i|}\tau\}\leq c$ is satisfied.
To see that, let us define
\begin{align*}
I_i:=&
\int_{A_i}
\frac{ f^p}{\left(\det \nabla (\phi_{i-1} \circ \ldots \circ \phi_1)\right)^{p-1}}~dx.
\end{align*}
Hence, as by the definition of $A_i$ (see \eqref{definition.M.j}) we obviously have $A_i\subset A_{i-1}$, using (\ref{main.det.on.M.i}) with $j=i-1$ we deduce
\begin{align*}
I_i=&
\int_{A_i}
\frac{1}{(\det \nabla \phi_{i-1}(\phi_{i-2}\circ \ldots \circ \phi_1))^{p-1}}
\cdot
\frac{ f^p}{\left(\det \nabla (\phi_{i-2} \circ \ldots \circ \phi_1)\right)^{p-1}} ~dx\\
\leq&
\int_{A_{i-1}}
\frac{1}{(\det \nabla \phi_{i-1}(\phi_{i-2}\circ \ldots \circ \phi_1))^{p-1}}
\cdot
\frac{ f^p}{\left(\det \nabla (\phi_{i-2} \circ \ldots \circ \phi_1)\right)^{p-1}} ~dx
\\
\leq&
\frac{1}{(1+\tau_0)^{p-1}}
\int_{A_{i-1}}
\frac{ f^p}{\left(\det \nabla (\phi_{i-2} \circ \ldots \circ \phi_1)\right)^{p-1}} ~dx
\\
=&
\frac{1}{(1+\tau_0)^{p-1}} I_{i-1}.
\end{align*}
Thus
\begin{align*}
I_i\leq \frac{1}{(1+\tau_0)^{(p-1)(i-1)}} \int_\Omega f^p ~dx.
\end{align*}
Moreover
\begin{align*}
|M_i|=&|\phi_{i-1}\circ\cdots\circ \phi_{1}(A_i)|=2^p\int_{A_i}(1/2)^p\det\nabla(\phi_{i-1}\circ\cdots\circ \phi_{1})~dx\\
\leq&2^p\int_{A_i}\frac{f^p}{(\det \nabla (\phi_{i-1}\circ \ldots \circ \phi_1))^p}\det\nabla(\phi_{i-1}\circ\cdots\circ \phi_{1})~dx\\
=&2^pI_i
\end{align*}
and thus
\begin{align}\label{decay.M.i}
 |M_i|\leq \frac{\|2f\|_{L^p}^{p}}{(1+\tau_0)^{(p-1)(i-1)}}\leq \|2f\|_{L^p}^{p}.
\end{align}
Choosing the upper bound $\tilde c$ for $\|f\|_{L^p}$ in the assumptions of the proposition small enough (independently of $i$), we therefore can apply Theorem \ref{BiLipschitzMapTheorem} to $\tau=\tau_0$ and $M=M_i$ to get $\phi_i.$\smallskip

\textit{Step 1.3 (additional estimates).}
By (\ref{decay.M.i}) we deduce that
\begin{align}
\sum_{i=1}^{\infty} |M_i|^{1/2}\leq&\sum_{i=1}^{\infty}\frac{\|2f\|_{L^p}^{p/2}}{(1+\tau_0)^{(p-1)(i-1)/2}}
=\frac{\|2f\|_{L^p}^{p/2}}{1-(1+\tau_0)^{-(p-1)/2}}.\label{M.k.in.l.2}
\end{align}
Next, using (\ref{main.det.global.i}) and (\ref{M.k.in.l.2}), we have for every $m\geq i\geq 1$
\begin{align}\det\nabla (\phi_m\circ\cdots\circ \phi_i)
=&\det \nabla \phi_m (\phi_{m-1}\circ\ldots \phi_i) \cdot
\ldots \cdot \det \nabla \phi_{i+1}(\phi_i) \cdot \det \nabla \phi_i\nonumber
\\
\geq&\prod_{k=i}^m (1-C\tau_0|M_k|^{1/2})\geq 1-C\tau_0\sum_{k=i}^{m}|M_k|^{1/2}\nonumber\\
\geq& 1-C\tau_0\frac{\|2f\|_{L^p}^{p/2}}{1-(1+\tau_0)^{-(p-1)/2}}.\label{geq.one.half}
\end{align}
In particular, choosing the upper bound $\tilde c$ for $\|f\|_{L^p}$ smaller if necessary, we deduce from the previous inequality that for every $m\geq i\geq 1$
\begin{align}\label{det.nabla.u.m....u.1.geg.half}
\det\nabla (\phi_m\circ\cdots\circ \phi_i)\geq 1/2\quad \text{a.e in $\Omega.$}
\end{align}
Finally, taking again $\tilde c$ smaller if necessary (note that this smallness will only depend on $p,q$ and $\Omega$), we get using (\ref{main.det.global.i}), (\ref{decay.M.i}), and (\ref{M.k.in.l.2})
\begin{align}&\prod_{k=1}^{m}\left\|\frac{1}{\det\nabla \phi_k}\right\|_{L^{\infty}}\leq \prod_{k=1}^{m}\frac{1}{1-C\tau_0|M_k|^{1/2}}=\prod_{k=1}^{m}\left(1+\frac{C\tau_0|M_k|^{1/2}}{1-C\tau_0|M_k|^{1/2}}\right)\nonumber\\
\leq&\prod_{k=1}^{m}\left(1+2C\tau_0|M_k|^{1/2}\right)
\leq \exp\left(\sum_{k=1}^{m}2C\tau_0|M_k|^{1/2}\right)\leq 2.\label{estimate.prod.1.over.det}
\end{align}

\textit{Step 2 (properties of the limiting mapping).} In the remaining steps we will show that
\begin{align*}
\phi:=\lim_{m\rightarrow \infty}\phi_m\circ\cdots\circ\phi_1
\end{align*}
is well-defined and has all the desired properties, namely that $\phi:\overline{\Omega}\rightarrow \overline{\Omega}$ is a homeomorphism
with $\phi,\phi^{-1}\in W^{1,q}(\Omega;\Omega)$ for which (\ref{equation.det.gen}), (\ref{est.det.simpl}), and (\ref{est.psi.w.1.q}) hold.\smallskip

\textit{Step 2.1.} We first prove that $\{\nabla(\phi_m\circ\cdots\circ \phi_1)\}_{m\in \mathbb{N}}$ is a Cauchy sequence in $L^{q}.$
First for every integer $n$,
\begin{align*}
&\|\nabla (\phi_{n+1} \circ \ldots \circ \phi_1)-\nabla (\phi_n \circ \ldots \circ \phi_1)\|_{L^{q}}^q
\\
&\leq \int_{\Omega}\left|\nabla \phi_{n+1}(\phi_n\circ \ldots \circ \phi_1)-\Id\right|^q
\left|\nabla (\phi_{n}\circ \ldots \circ \phi_1)\right|^q~dx
\\
&\leq \int_{\Omega}\left|\nabla \phi_{n+1}(\phi_n\circ \ldots \circ \phi_1)-\Id\right|^q
\prod_{k=1}^n\left|\nabla \phi_{k}(\phi_{k-1}\circ\ldots\circ \phi_1)\right|^q~dx
\\
&=\int_{\Omega}\left|\nabla \phi_{n+1}(\phi_n\circ \ldots \circ \phi_1)-\Id\right|^q
\left(\prod_{k=1}^n\frac{\left|\nabla \phi_{k}(\phi_{k-1}\circ\ldots\circ \phi_1)\right|^q}
{\det \nabla \phi_k (\phi_{k-1}\circ \ldots \circ \phi_1)}\right)
\det \nabla (\phi_{n}\circ \ldots \circ \phi_1)~dx
\\
&\overset{\eqref{main.nabla.phi.bounded.by.det.i},\eqref{main.W.1.p.i}}{\leq~~~~~~}
\int_{\Omega}\left|\nabla \phi_{n+1}
(\phi_n\circ \ldots \circ \phi_1)-\Id\right|^q\left(\prod_{k=1}^nC(1+C\tau_0)^{q-1}\right)
\det \nabla (\phi_{n}\circ \ldots \circ \phi_1)~dx
\\
&=\int_{\Omega}\left|\nabla \phi_{n+1}-\Id\right|^qC^{n} (1+C\tau_0)^{n(q-1)}~dx
\\
&\overset{\eqref{main.W.1.p.i}}{\leq\,}
(C\tau_0)^q|M_{n+1}|^{1/2}
C^{n} (1+C\tau_0)^{n(q-1)}
\\
&\overset{\eqref{tau.big.enough}}{\leq\,}
(C\tau_0)^q|M_{n+1}|^{1/2}
(1+\tau_0)^{n(1+\lambda)(q-1)}
\\
&\overset{\eqref{decay.M.i}}{\leq\,}
(C\tau_0)^q\|2f\|_{L^p}^{p/2}(1+\tau_0)^{n[(1+\lambda)(q-1)-(p-1)/2]}.
\end{align*}
Taking the $q$-th root in the previous inequality we get that, for any $m\geq n+1$,
\begin{align}
&\|\nabla (\phi_{m} \circ \ldots \circ \phi_1)-\nabla (\phi_n \circ \ldots \circ \phi_1)\|_{L^{q}}\nonumber\\
\leq &\sum_{k=n}^{m-1}\|\nabla (\phi_{k+1} \circ \ldots \circ \phi_1)-\nabla (\phi_k \circ \ldots \circ \phi_1)\|_{L^{q}}\nonumber\\
\leq & C\tau_0\|2f\|_{L^p}^{p/(2q)}\sum_{k=n}^{m-1}(1+\tau_0)^{k[(1+\lambda)(q-1)-(p-1)/2]/q}\label{est.nabla.phi.m.W.1.q}
\end{align}
which proves that
$\{\nabla(\phi_m\circ\cdots\circ \phi_1)\}_{m\in \mathbb{N}}$ is a Cauchy sequence in $L^{q}$ since by (\ref{relation.q.p}) we know that
\begin{align*}
 (1+\lambda)(q-1)-(p-1)/2<0.
\end{align*}

\textit{Step 2.2.}
Since $\phi_k=\operatorname{id}$ on $\partial \Omega$ for every $k$, we deduce from the Poincar\'e inequality and from (\ref{est.nabla.phi.m.W.1.q}) that
$(\phi_m\circ\cdots\circ \phi_1)_{m\in \mathbb{N}}$ is a Cauchy sequence in $W^{1,q}(\Omega)$. Hence $\phi_m\circ\cdots\circ\phi_1$ converges to some $\phi$ in $W^{1,q}$ as $m\rightarrow \infty$.
Using (\ref{est.nabla.phi.m.W.1.q}) with $n=0$ and again the Poincar\'e inequality, we directly get (\ref{est.psi.w.1.q}).
By (\ref{main.L.infty.i}), we have for every $m\geq n+1$
\begin{align*}
\|\phi_{m} \circ \ldots \circ \phi_1-\phi_n \circ \ldots \circ \phi_1\|_{L^{\infty}}
\leq &\sum_{k=n}^{m-1}\|\phi_{k+1} -\operatorname{id}\|_{L^{\infty}}
\leq \sum_{k=n}^{m-1}C\tau_0|M_{k+1}|^{1/2}.
\end{align*}
Using (\ref{M.k.in.l.2}), this implies that $(\phi_{m} \circ \ldots \circ \phi_1)_m$ is a Cauchy sequence in $L^{\infty}$. Thus, $\phi_{m} \circ \ldots \circ \phi_1$ converges to $\phi$ uniformly and we have $\phi\in C^{0}(\overline{\Omega})$.
\smallskip

\textit{Step 2.3 (existence and integrability of $\phi^{-1}$).}
First, for every $m\in \mathbb{N}$, we have (recall that we are working in 2 dimensions)
\begin{align*}
&||\nabla [(\phi_m \circ \ldots \circ \phi_1)^{-1}]||_{L^q}=||(\nabla (\phi_m \circ \ldots \circ \phi_1))^{-1}((\phi_m\circ\cdots\circ \phi_1)^{-1})||_{L^q}\\
=&\left(\int_{\Omega}\left|\frac{\operatorname{adj}(\nabla (\phi_m \circ \ldots \circ \phi_1))}{\det\nabla (\phi_m \circ \ldots \circ \phi_1)}\right|^q((\phi_m\circ\cdots\circ \phi_1)^{-1})\,dx\right)^{1/q}
\\
=&\left(\int_{\Omega}\left|\frac{\nabla (\phi_m \circ \ldots \circ \phi_1)}{\det\nabla (\phi_m \circ \ldots \circ \phi_1)}\right|^q((\phi_m\circ\cdots\circ \phi_1)^{-1})\,dx\right)^{1/q}
\\
=&\left(\int_{\Omega}\frac{|\nabla (\phi_m \circ \ldots \circ \phi_1)|^q}{|\det\nabla (\phi_m \circ \ldots \circ \phi_1)|^{q-1}}\,dx\right)^{1/q}
\\
\leq& ||\nabla (\phi_m \circ \ldots \circ \phi_1)||_{L^q}\left(\prod_{i=1}^m\left\|\frac{1}{\det\nabla \phi_i}\right\|_{L^{\infty}}\right)^{(q-1)/q}.
\end{align*}
Using (\ref{estimate.prod.1.over.det}) we directly deduce from the previous inequality that the sequence $((\phi_m\circ \ldots \circ \phi_1)^{-1})_{m\in \mathbb{N}}$ is bounded in $W^{1,q}$ (uniformly in $m$) and hence converges weakly (up to a subsequence) to some $\xi\in W^{1,q}$. Moreover since trivially
\begin{align*}
\|\phi_j^{-1}-\operatorname{id}\|_{L^{\infty}}=\|\phi_j-\operatorname{id}\|_{L^{\infty}}
\end{align*}
proceeding exactly as Step 2.2 we know that
$\xi\in C^0(\overline{\Omega})$
and
\begin{align*}
\|(\phi_{m} \circ \ldots \circ \phi_1)^{-1}-\xi\|_{L^{\infty}}\rightarrow 0\quad \text{as $m\rightarrow \infty.$}
\end{align*}
Moreover from the uniform convergence we get that, for every $x\in \overline{\Omega}$,
\begin{align*}
x=\lim_{m\rightarrow \infty}(\phi_{m} \circ \ldots \circ \phi_1)\circ(\phi_{m} \circ \ldots \circ \phi_1)^{-1}(x)=(\phi\circ \xi)(x)\end{align*}
and similarly $x=(\xi\circ\phi)(x)$. Hence $\xi=\phi^{-1}$ and $\phi:\overline{\Omega}\rightarrow \overline{\Omega}$ is an homeomorphism.
Finally, since for every $m$ we have $\phi_{m} \circ \ldots \circ \phi_1=\operatorname{id}$ on $\partial \Omega$, passing to the limit we get
\begin{align*}
\phi=\operatorname{id}\quad \text{on $\partial \Omega.$}
\end{align*}

\textit{Step 2.4 (proving the assertions about $\det\nabla \phi$).}  We now show that $\det\nabla \phi\geq  f$ in $\Omega$ together with (\ref{est.det.simpl}), where, as always, the inequalities are understood in the weak sense for $q<2$ and in the pointwise a.\,e.\ sense for $q\geq 2$.
We first claim that we have a.\,e.\ in $\Omega$
\begin{align}\label{det.up.to.m}
\det\nabla(\phi_m\circ\cdots\circ \phi_1)\geq f\chi_{\Omega\setminus A_m}
\end{align}
where we recall that
\begin{align*}
A_m= \bigcap_{0\leq k\leq m-1}
\left\{ \frac{ f}{\det \nabla (\phi_{k}\circ \ldots \circ \phi_1)}\geq \frac{1}{2}\right\}
\end{align*}
with the convention that $A_1=\{ f\geq 1/2\}$ and $A_0=\Omega.$
Indeed  noticing  that, by the very definition of the sets $A_m$,
$$\det
\nabla (\phi_{i-1}\circ\cdots\circ \phi_1)>2 f\quad \text{on $A_{i-1}\setminus A_i$},$$
and using (\ref{det.nabla.u.m....u.1.geg.half}),
we get that, a.\,e.\ in $\Omega,$ and for every $1\leq i\leq m$,
\begin{align*}
\det\nabla(\phi_m\circ\cdots\circ \phi_1)&=\det\nabla(\phi_m\circ\cdots\circ \phi_i)(\phi_{i-1}\circ\cdots\circ \phi_1)\cdot \det\nabla (\phi_{i-1}\circ\cdots\phi_1)\\
&\geq \det\nabla(\phi_m\circ\cdots\circ \phi_i)(\phi_{i-1}\circ\cdots\circ \phi_1)\cdot 2f \chi_{A_{i-1}\setminus A_i}\\
&\geq f \chi_{A_{i-1}\setminus A_i},
\end{align*}
from which (\ref{det.up.to.m}) directly follows.
Moreover, using (\ref{definition.M.j}), (\ref{decay.M.i}) and (\ref{estimate.prod.1.over.det}), we obtain
\begin{align}\label{A.i.tends.to.zero}|A_i|=|\phi_1^{-1}\circ\cdots\circ \phi_{i-1}^{-1}(M_i)|\leq 2|M_i|\rightarrow 0\quad \text{as $i\rightarrow \infty$}.
\end{align}
On one hand, recalling that $\phi_m\circ\cdots\circ\phi_1$ converges strongly to $\phi$ both  in $W^{1,q}$ and  $L^{\infty}$ we get that, for every $\eta\in C^{\infty}_{cpt}(\Omega)$,
\begin{align*}\int_{\Omega}\det[\nabla(\phi_m\circ\cdots\circ\phi_1)]\eta ~dx &=
-\frac{1}{2}\int_\Omega \langle\adj \nabla(\phi_m\circ\cdots\circ\phi_1) \cdot \phi_m\circ\cdots\circ\phi_1), \nabla \eta\rangle ~dx
\\
&\rightarrow -\frac{1}{2}\int_\Omega \langle\adj \nabla\phi \cdot \phi, \nabla \eta \rangle ~dx\quad \text{as $m\rightarrow \infty$}
\end{align*} or, equivalently, using our notation,
\begin{align*}\mathcal{J}_{\phi_m\circ\cdots\circ\phi_1}[\eta]\rightarrow \mathcal{J}_{\phi}[\eta]\quad \text{as $m\rightarrow \infty$}.
\end{align*}
On the other hand, using (\ref{det.up.to.m}) and (\ref{A.i.tends.to.zero}) we obtain for every nonnegative $\eta \in C^{\infty}_{cpt}(\Omega)$,
\begin{align*}\int_{\Omega}\det[\nabla(\phi_m\circ\cdots\circ\phi_1)]\eta\geq \int_{\Omega}f\chi_{\Omega\setminus A_m}\eta ~dx\rightarrow \int_{\Omega}f\eta ~dx.
\end{align*}
Combining the two previous equations we get that (in the sense of the distribution for $q<2$ and in the a.\,e.\ pointwise sense for $q\geq 2$)
\begin{align*}
\det\nabla \phi\geq f.
\end{align*}
Finally (\ref{est.det.simpl}) is a direct consequence of (\ref{geq.one.half}). This concludes the proof.\smallskip
\end{proof}

With the help of the previous result we now establish Theorem \ref{L.p}.

\begin{proof}[Proof of Theorem \ref{L.p}]
\textit{Step 1 (preliminaries).}
Fix $\delta>0$ small enough so that (recall that by assumption $\int_\Omega f<|\Omega|$)
\begin{align}
(1+\delta)\int_{\Omega}f~dx<|\Omega|(1-\delta)\label{choice.delta}
\end{align}
and let (extending $f$ by $0$ outside $\Omega$)
\begin{align*}
f_{\epsilon}:=(1+\delta)\rho_{\epsilon}\ast f+l_{\epsilon}
\end{align*}
where $\rho_{\epsilon}$ is the usual convolution kernel and
where $l_{\epsilon}\in \mathbb{R}$ is chosen so that
\begin{align}
\int_{\Omega}f_{\epsilon}\,dx=|\Omega|\label{int.cond.f.epsilon}.
\end{align}
Note that by (\ref{choice.delta}) we have
\begin{align}\label{l.epsilon.geg.delta}l_{\epsilon}\geq \delta.
\end{align}
Hence, since $f_{\epsilon}\in C^{\infty}(\overline{\Omega}),$ since $f_{\epsilon}\geq l_{\epsilon}\geq \delta>0$ and since the boundary of $\Omega$ is $C^{1,1}$, by (\ref{int.cond.f.epsilon}) there exists (see Theorem 5 in \cite{Dac-Moser}) a map $\psi_{\epsilon}\in \operatorname{Diff}^{1,1}(\overline{\Omega};\overline{\Omega})$ satisfying
\begin{subequations}
\label{equation.for.psi.epsilon}
\begin{align}
\det\nabla \psi_{\epsilon}&=f_{\epsilon} &&\text{in }\Omega,
\\
\psi_{\epsilon}&=\operatorname{id} &&\text{on }\partial \Omega.
\end{align}
\end{subequations}
Define the measurable set $A_{\epsilon}\subset \Omega$ by
\begin{align*}A_{\epsilon}:=\{x\in \Omega: f_{\epsilon}(x)\leq (1+\delta)f(x)\}.
\end{align*}
Using (\ref{l.epsilon.geg.delta}) and the definition of $f_{\epsilon}$ we deduce that
\begin{align}\label{A.epsilon.goes.to.0}
|A_{\epsilon}|\rightarrow 0\quad \text{as $\epsilon\rightarrow 0$.}
\end{align}

\textit{Step 2.}
We claim that, for every $\epsilon>0$ small enough, we can find a homeomorphism $\varphi_{\epsilon}:\overline{\Omega}\rightarrow \overline{\Omega}$ with
$\varphi_{\epsilon},\varphi_{\epsilon}^{-1}\in W^{1,q}(\Omega;\Omega)$ satisfying
\begin{subequations}
\label{equation.for.varphi.epsilon}
\begin{align}
\det\nabla \varphi_{\epsilon}
&\geq
\frac{f(\psi_{\epsilon}^{-1})}{f_{\epsilon}(\psi_{\epsilon}^{-1})}\chi_{A_{\epsilon}}(\psi_{\epsilon}^{-1})&& \text{in }\Omega,
\\
\varphi_{\epsilon}&=\operatorname{id} && \text{on }\partial \Omega,
\end{align}
\end{subequations}
and
\begin{align}\label{borne.inf.det.varphi.eps}
\det\nabla \varphi_{\epsilon}&\geq \frac{1}{1+\delta} &&\text{in }\Omega,
\end{align}
where both inequalities are understood in the pointwise a.\,e.\ sense for $q\geq 2$ and in the weak sense for $q<2$.
Indeed using $\psi_{\epsilon}$ as a change of variables, recalling that $f_{\epsilon}\geq \delta$ in $\Omega$ and using (\ref{A.epsilon.goes.to.0}) we get
\begin{align*}
\int_{\Omega}\left|\frac{f(\psi_{\epsilon}^{-1})}{f_{\epsilon}(\psi_{\epsilon}^{-1})}\chi_{A_{\epsilon}}(\psi_{\epsilon}^{-1})\right|^p~dx
&=\int_{A_{\epsilon}}\frac{|f|^p}{|f_{\epsilon}|^{p-1}}~dx\\
&\leq \frac{1}{\delta^{p-1}}\|f\|_{L^p(A_{\epsilon})}^p\rightarrow 0\quad \text{as $\epsilon\rightarrow 0$.}
\end{align*}
Hence we can use Proposition \ref{proposition.assuming.f.L.p.small} (with the right-hand side equal to $\frac{f(\psi_{\epsilon}^{-1})}{f_{\epsilon}(\psi_{\epsilon}^{-1})}\chi_{A_{\epsilon}}(\psi_{\epsilon}^{-1})$) and get the claim. Note that in particular (\ref{borne.inf.det.varphi.eps}) follows directly from (\ref{est.det.simpl}) and (\ref{A.epsilon.goes.to.0}).
\smallskip

\textit{Step 3 (conclusion).}
We claim that, for $\epsilon$ small enough,
\begin{align*}
\phi=\varphi_{\epsilon}\circ \psi_{\epsilon}
\end{align*}
satisfies all the wished properties.
First,
$\phi:\overline{\Omega}\rightarrow \overline{\Omega}$ is an homeomorphism, is identically the identity on $\partial \Omega$ and $\phi,\phi^{-1}\in W^{1,q}(\Omega;\Omega).$
It remains to show
\begin{align}\label{det.phi.step.1.3}
\det\nabla \phi\geq f \quad\quad\text{in $\Omega$},
\end{align}
where the inequality is required to hold pointwise almost everywhere for $q\geq 2$ and where the inequality is understood in the weak sense for $q<2$.

We first prove (\ref{det.phi.step.1.3}) when $q\geq 2$ (in order to expose the idea more clearly).
On one hand, using (\ref{equation.for.psi.epsilon}) and (\ref{equation.for.varphi.epsilon}), we get a.\,e.\ in $\Omega$
\begin{align*}
\det\nabla\phi
=\det\nabla\varphi_{\epsilon}(\psi_{\epsilon})\det\nabla\psi_{\epsilon}
\geq f\chi_{A_{\epsilon}}
\end{align*}
On the other hand, by definition of $A_{\epsilon}$ and by (\ref{borne.inf.det.varphi.eps}) we get
a.\,e.\ in $\Omega$
\begin{align*}
\det\nabla\phi
=\det\nabla\varphi_{\epsilon}(\psi_{\epsilon})\det\nabla\psi_{\epsilon}
\geq \frac{f_{\epsilon}}{1+\delta}
\geq f\chi_{\Omega\setminus A_{\epsilon}}.
\end{align*}
The combination of the previous two inequalities gives (\ref{det.phi.step.1.3}).

We finally prove (\ref{det.phi.step.1.3}) when $q<2$, namely
\begin{align*} \mathcal{J}_{\phi}[\eta]=-\frac{1}{2}\int_{\Omega}\langle \operatorname{adj}\nabla \phi\cdot\phi,\nabla \eta\rangle~dx\geq \int_{\Omega}f\eta ~dx\quad\text{for every $\eta\in C^{\infty}_{cpt}(\Omega;[0,\infty))$}.
\end{align*}
By classical properties of $\operatorname{adj}$ and changing the variables we get
\begin{align*}
&\int_{\Omega}\langle\operatorname{adj}\nabla\phi\cdot \phi,\nabla \eta\rangle ~dx
\\
=&\int_{\Omega}\langle\operatorname{adj}\left[\nabla\varphi_{\epsilon}(\psi_{\epsilon})\cdot\nabla \psi_{\epsilon}\right]\cdot \varphi_{\epsilon}(\psi_{\epsilon}),\nabla \eta\rangle ~dx\\
=&\int_{\Omega}\frac{\langle\operatorname{adj}\left[\nabla\varphi_{\epsilon}\cdot\nabla \psi_{\epsilon}(\psi_{\epsilon}^{-1})\right]\cdot \varphi_{\epsilon},\nabla \eta(\psi_{\epsilon}^{-1})\rangle}{\det\nabla \psi_{\epsilon}(\psi_{\epsilon}^{-1})} ~dx\\
=&\int_{\Omega}\frac{\langle\operatorname{adj}\left[\nabla\psi_{\epsilon}(\psi_{\epsilon}^{-1})\right]\cdot\operatorname{adj}
\left[\nabla\varphi_{\epsilon}\right]\cdot \varphi_{\epsilon},\nabla \eta(\psi_{\epsilon}^{-1})\rangle}{\det\nabla \psi_{\epsilon}(\psi_{\epsilon}^{-1})} ~dx\\
=&\int_{\Omega}\frac{\langle\operatorname{adj}\left[\nabla\psi_{\epsilon}(\psi_{\epsilon}^{-1})\right]\cdot\operatorname{adj}
\left[\nabla\varphi_{\epsilon}\right]\cdot \varphi_{\epsilon},(\nabla\psi_{\epsilon})^t(\psi_{\epsilon}^{-1})\cdot\nabla [\eta\circ\psi_{\epsilon}^{-1}])\rangle}{\det\nabla \psi_{\epsilon}(\psi_{\epsilon}^{-1})} ~dx
\\
=&\int_{\Omega}\frac{\langle\operatorname{adj}
\left[\nabla\varphi_{\epsilon}\cdot \right]\cdot \varphi_{\epsilon},\operatorname{adj}\left[\nabla\psi_{\epsilon}(\psi_{\epsilon}^{-1})\right]^t\cdot
(\nabla\psi_{\epsilon})^t(\psi_{\epsilon}^{-1})\cdot\nabla [\eta\circ\psi_{\epsilon}^{-1}])\rangle}{\det\nabla \psi_{\epsilon}(\psi_{\epsilon}^{-1})} ~dx
\\
=&\int_{\Omega}\langle\operatorname{adj}
\left[\nabla\varphi_{\epsilon} \right]\cdot \varphi_{\epsilon}, \nabla[\eta\circ\psi_{\epsilon}^{-1}]\rangle ~dx.
\end{align*}
Then, for every $\rho\in C^{\infty}(\overline{\Omega};[0,1])$ we get, using (\ref{equation.for.varphi.epsilon}) and (\ref{borne.inf.det.varphi.eps}) in the weak sense and then changing the variables,
\begin{align*}
&-\frac{1}{2}\int_{\Omega}\langle\operatorname{adj}
\left[\nabla\varphi_{\epsilon}\right]\cdot \varphi_{\epsilon}, \nabla[\eta\circ\psi_{\epsilon}^{-1}]\rangle ~dx
\\
=&-\frac{1}{2}\int_{\Omega}\langle\operatorname{adj}
\left[\nabla\varphi_{\epsilon} \right]\cdot \varphi_{\epsilon}, \nabla[(\rho\eta)\circ\psi_{\epsilon}^{-1}]\rangle ~dx-\frac{1}{2}\int_{\Omega}\langle\operatorname{adj}
\left[\nabla\varphi_{\epsilon} \right]\cdot \varphi_{\epsilon}, \nabla[((1-\rho)\eta)\circ\psi_{\epsilon}^{-1}]\rangle ~dx
\\
\geq &\int_{\Omega} \left(\frac{f}{\det\nabla \psi_{\epsilon}}\rho\eta\chi_{A_{\epsilon}}
\right)\circ\psi_{\epsilon}^{-1}~dx+\int_{\Omega} \frac{1}{1+\delta}((1-\rho)\eta)\circ\psi_{\epsilon}^{-1}~dx
\\
=&\int_{\Omega}f\rho\eta\chi_{A_{\epsilon}}~dx+\int_{\Omega} \frac{f_{\epsilon}}{1+\delta}(1-\rho)\eta~dx
\end{align*}
By approximation, the above equation is valid for $\rho=\chi_{A_{\epsilon}}$. Hence, by definition of $A_{\epsilon}$ we get
combining the last two equations,
\begin{align*}
&-\frac{1}{2}\int_{\Omega}\langle\operatorname{adj}\nabla\varphi\cdot \varphi,\nabla \eta\rangle ~dx
\geq \int_{\Omega}f\chi_{A_{\epsilon}}\eta~dx+\int_{\Omega} \frac{f_{\epsilon}}{1+\delta}\chi_{\Omega\setminus A_{\epsilon}}\eta~dx\\
\geq &\int_{\Omega}f\chi_{A_{\epsilon}}\eta~dx+\int_{\Omega} f\chi_{\Omega\setminus A_{\epsilon}}\eta~dx
=\int_{\Omega}f\eta ~dx,
\end{align*}
which proves (\ref{det.phi.step.1.3}) and ends the proof.

\end{proof}

\section{The prescribed Jacobian in the $L^{\infty}$ case}

In this section we prove Theorem \ref{theorem.det.nabla.u.geq.f}. This idea is the following: First, by convolution and classical
results for the prescribed Jacobian equation in the smooth case, we obtain a
smooth map $\varphi$ such that $\varphi=\operatorname{id}$ on $\partial \Omega$
and $\det\nabla \varphi> f$ outside a set $M$ of small measure.
Postcomposing $\varphi$ with the stretching map (from Theorem \ref{BiLipschitzMapTheorem}) $\phi_{\tau,\varphi(M)}$
for some well-chosen $\tau$ then yields the desired solution to the prescribed
Jacobian inequality.


\begin{proof}[Proof of Theorem \ref{theorem.det.nabla.u.geq.f}]~

\textit{Step 1 (sharp regularity).} A very simple example shows that there exists a nonnegative function $f\in L^{\infty}(\Omega)$ with $\int_{\Omega}f<|\Omega|$ such that no solution to (\ref{Jacobian.inequality.theorem}) can be of class $C^1$:
Let $M\subset \Omega$ be an open dense (in $\Omega$) set with $|M|<|\Omega|/2$
and let $f=2\chi_{M}.$ We argue by contradiction and assume that there exists a
solution $\phi\in C^1$ of (\ref{Jacobian.inequality.theorem}). Then we would have
$\det\nabla \phi\geq 2$ a.\,e.\ in $M$. By continuity of $\det\nabla \phi$ and the fact that $M$ is open and dense, the previous inequality would imply $\det\nabla \phi\geq 2$ everywhere in $\Omega$, which contradicts $\int_{\Omega}\det\nabla \phi\,dx=|\Omega|$.
\smallskip

Let us now turn to the existence part of our theorem.

\textit{Step 2 (preliminaries).} Define
\begin{align*}
\beta:=\frac{|\Omega|-\int_{\Omega}f\,dx}{|\Omega|}\in (0,1].
\end{align*}
Take $\tau_{0}$ big enough so that
\begin{equation} \label{tau.assez.grand}
\frac{(1+\tau_{0})\beta}{2}\geq \|f\|_{L^{\infty}}.
\end{equation}
Next choose $\epsilon_{0}>0$ small enough so that
\begin{equation} \label{epsilon.assez.petit}
\left(1-C\tau_{0}\sqrt{(\|f\|_{L^{\infty}}+1)\epsilon_{0}}\right)
\left(\frac{\beta}{2}+y\right)\geq y
\quad \text{for every $y\in [0,\|f\|_{L^{\infty}}];$}
\end{equation}
taking $\epsilon_{0}$ even smaller we can moreover assume that
\begin{equation}
\label{volume.assez.petit}
(\|f\|_{L^{\infty}}+1)\epsilon_{0},\sqrt{(\|f\|_{L^{\infty}}+1)\epsilon_{0}}\tau_0\leq c,
\end{equation}
where $C,c>0$ are the constants (depending only on $\Omega$) in the statement of
Theorem~\ref{BiLipschitzMapTheorem}.
\smallskip

\textit{Step 3 (approximation).}
Extending $f$ by $1$ outside of $\Omega$, mollifying the resulting function,
and adding an appropriate constant, it is elementary to construct
a sequence $f_{\nu}\in C^{\infty}(\overline{\Omega})$, $\nu\in \mathbb{N}$, such
that $\int_{\Omega}f_{\nu}\,dx=|\Omega|,$
\begin{equation}
\label{f.nu.sandwich}\beta/2\leq f_{\nu}\leq \|f\|_{L^{\infty}}+1\quad
\text{in $\Omega$},
\end{equation}
and
\begin{align*}
f_{\nu}(x)\rightarrow f(x)+\beta\quad
\text{for a.\,e.\ }x\in \Omega.
\end{align*}
The last formula implying convergence in measure,  there exist $\nu_0$ and a
measurable set $A\subset \Omega$ such that $|A|\leq \epsilon_0$ and
\begin{equation}
\label{f.N.plus.grand.que.f}
f_{\nu_0}\geq \frac{\beta}{2}+f\quad
\text{a.\,e.\ in $\Omega\setminus A.$}
\end{equation}
\textit{Step 4 (conclusion).}
By a classical result for the prescribed Jacobian equation (see Theorem 5 in
\cite{Dac-Moser}; recall that the boundary of $\Omega$ is of class $C^{1,1}$ and that $\int_{\Omega}f_{\nu_0}\,dx=|\Omega|$) there exists
$\varphi\in \operatorname{Diff}^{1,1}(\overline{\Omega};\overline{\Omega})$
so that
\begin{align*}
\det\nabla\varphi=f_{\nu_0}\quad \text{in $\Omega$} \quad \text{and}\quad
\varphi=\operatorname{id}\quad \text{on $\partial \Omega.$}
\end{align*}
Using (\ref{f.nu.sandwich}), we have
\begin{equation}
\label{estimate.volume}
|\varphi(A)|
=\int_{A}\det\nabla \varphi \,dx
\leq (\|f\|_{L^{\infty}}+1)|A|
\leq  (\|f\|_{L^{\infty}}+1)\epsilon_{0}.
\end{equation}
Hence, from the previous inequality and (\ref{volume.assez.petit}),
we can apply Theorem \ref{BiLipschitzMapTheorem} with
$M=\varphi(A)$ and $\tau=\tau_0$ and get a bi-Lipschitz mapping
$\psi:\overline{\Omega}\rightarrow \overline{\Omega}$ such that
$\psi=\operatorname{id}$ on $\partial \Omega$ and
\begin{align}
\det \nabla \psi&\geq 1+\tau_{0}\quad
\text{a.\,e.\ in $\varphi(A),$}\label{det.on.vaphi.K}
\\
\det \nabla \psi&\geq 1-C|\varphi(A)|^{1/2}\tau_{0}\quad
\text{a.\,e.\ in $\Omega.$}\label{det.varphi.global}
\end{align}
We claim that $\phi:=\psi\circ \varphi$ has all the desired properties. First we
obviously have $\phi=\operatorname{id}$ on $\partial \Omega.$ It remains to
show that
\begin{align*}
\det\nabla \phi=\det\nabla\psi(\varphi)\cdot f_{\nu_0}\geq f\quad \text{a.e in $\Omega.$}
\end{align*}
First, using (\ref{det.on.vaphi.K}), (\ref{f.nu.sandwich}) and (\ref{tau.assez.grand}), we obtain that a.\,e.\ in $A$
\begin{align*}
\det\nabla\psi(\varphi)\cdot f_{\nu_0}\geq (1+\tau_{0})f_{\nu_0}\geq (1+\tau_{0})\frac{\beta}{2}\geq\|f\|_{L^{\infty}}.
\end{align*}
Finally, using (\ref{det.varphi.global}), (\ref{estimate.volume}),
(\ref{f.N.plus.grand.que.f}), and (\ref{epsilon.assez.petit}), we get that a.\,e.\ in $\Omega\setminus A$
\begin{align*}
\det\nabla\psi(\varphi)\cdot f_{\nu_0}
&\geq (1-C\tau_0\sqrt{|\varphi(A)|})f_{\nu_0}
\geq (1-C\tau_0\sqrt{(\|f\|_{L^{\infty}}+1)\epsilon_{0}})f_{\nu_0}\\
&\geq (1-C\tau_0\sqrt{(\|f\|_{L^{\infty}}+1)\epsilon_{0}})
\left(\frac{\beta}{2}+f\right)
\geq f,
\end{align*}
which ends the proof.
\end{proof}

\section{Bi-Lipschitz stretching of a given set with small
measure}\label{Section.Bi-Lip.Mapping}

In this section we establish Theorem \ref{BiLipschitzMapTheorem}. The proof will
consists of two main parts. In the first part, contained in the Sections \ref{sub.sect.Simplification.of.the.domain}, \ref{sub.sect.Simplification.of.the.set.M}, and \ref{sub.sect.Simplification.of.the.boundary.values}, we reduce the problem to the case $\Omega=(0,1)^2$ (see Lemma \ref{Lemma.simpl.domain}) with a compact subset $M\subset \Omega$ (see Lemma \ref{Lemma.simpl.set}); furthermore, we drop the requirement of identity boundary conditions, replacing it by a global preservation of the boundary (see Lemma \ref{lemma.simpl.boundary.values}). This way we will only be left with proving
the simplified version of the existence of stretching maps that is stated in Proposition \ref{Prop.Bi.Lip.simpl} below.

In the second part -- see the Sections \ref{SubSection.covering} and \ref{subsection.stretching} -- , we prove Proposition \ref{Prop.Bi.Lip.simpl}.  To this end we first establish a covering result (see Lemma \ref{IntersectionFreeCovering}) based on results by Alberti, Cs\"ornyei, and Preiss \cite{Alberti}: Any compact set $K\subset (0,1)^2$ can be covered
by a finite number a strips generated by $1$-Lipschitz functions in both axis
directions with the two following properties:
\begin{itemize}
\item The total width of the strips do not exceed $2\sqrt{|K|};$ \smallskip

\item The $x$-strips have pointwise disjoint interior and similarly for the $y$-strips.
\end{itemize}
With the help of this covering result we finally prove Proposition \ref{Prop.Bi.Lip.simpl} with an explicit formula for the stretching map.

\subsection{Simplification of the domain $\Omega$.}\label{sub.sect.Simplification.of.the.domain}

\begin{lemma}\label{Lemma.simpl.domain}
Theorem \ref{BiLipschitzMapTheorem} in the case of a general planar domain $\Omega\subset\mathbb{R}^2$ is a consequence of Theorem \ref{BiLipschitzMapTheorem} for the special case of the unit square $\Omega=(0,1)^2.$
\end{lemma}

\begin{proof}
The proof is based on two main ingredients, namely
\begin{itemize}
\item the fact that (\ref{main.boundary})-(\ref{main.det.global}) behave well under fixed $C^{1,1}$-diffeomorphisms, and
\item the fact that any $C^{1,1}$-domain $\Omega$ may be decomposed into finitely many subdomains $\Omega_i$, each of which may be mapped onto the unit square $(0,1)^2$ by a $C^{1,1}$-diffeomorphism.
\end{itemize}
So we assume that Theorem \ref{BiLipschitzMapTheorem} has been proven for
$\Omega=(0,1)^2$ and show how  to deduce it for any bounded open set
$\Omega\subset \mathbb{R}^2$ with $C^{1,1}$ boundary. To this end we recall that
any bi-Lipschitz mapping in the plane $\psi$ satisfies
\begin{equation}
\label{control.area}
\frac{1}{L^2}|M|\leq |\psi(M)|\leq L^2|M|
\end{equation}
whenever
$|x-y|/L\leq |\psi(x)-\psi(y)|\leq L|x-y|$ holds for every $x,y.$
\smallskip

\textit{Step 1.} We assume first that there exists a
bijection $\psi:[0,1]^2\rightarrow \overline{\Omega}$ such that $\psi\in
C^{1,1}([0,1]^2;\overline{\Omega})$ and $\psi^{-1}\in
C^{1,1}(\overline{\Omega};[0,1]^2)$ hold. We call such a set $\Omega$
\textit{$C^{1,1}$-equivalent to the unit square}.

Let $\tau>0$ and $M\subset \Omega$
with $|M|$ and $\sqrt{|M|}\tau$ small enough, which trivially  implies (by
\eqref{control.area}) that $|\psi^{-1}(M)|$ and $\sqrt{|\psi^{-1}(M)|}\tau$ are
also small. Hence, by hypothesis applied to  $2\tau$ in place of
$\tau$, there exists a bi-Lipschitz map $\widetilde{\phi}:[0,1]^2\rightarrow
[0,1]^2$ stretching the set $\psi^{-1}(M)\subset (0,1)^2,$ such that
\begin{align}
\widetilde{\phi}&=\operatorname{id}\quad
\text{on $\partial[0,1]^2,$}\label{simpl.domain.boundary}
\\
\|\widetilde{\phi}-\operatorname{id}\|_{W^{1,p}([0,1]^2)}&
\leq 2C\tau|\psi^{-1}(M)|^{1/(2p)}\quad
\text{for all $1\leq p\leq \infty,$}\label{simpl.domain.W.1.p}
\\
\|\widetilde{\phi}-\operatorname{id}\|_{L^\infty([0,1]^2)}&
\leq 2C\tau|\psi^{-1}(M)|^{1/2},
\label{simpl.domain.L.infty}
\\
|\nabla\widetilde{\phi}|&\leq C\det\nabla \widetilde{\phi}\quad\text{a.\,e.\ in $(0,1)^2,$}\label{simpl.domain.nabla.borne.par.det}\\
\det \nabla \widetilde{\phi}&\geq 1+2\tau\quad \text{a.\,e.\ in $\psi^{-1}(M),$}
\label{simpl.domain.det.on.M}
\\
\det \nabla \widetilde{\phi}&\geq 1-2C\tau|\psi^{-1}(M)|^{1/2}\quad
\text{a.\,e.\ in $(0,1)^2.$}
\label{simpl.domain.det.global}
\end{align}
We claim that
\begin{align*}
\phi:=\psi\circ \widetilde{\phi}\circ \psi^{-1}
\end{align*}
satisfies (\ref{main.boundary})-(\ref{main.det.global}).
Indeed first (\ref{main.boundary}) follows trivially from
(\ref{simpl.domain.boundary}).
In what follows $C_{\psi}$ will denote a generic constant depending only on
$\|\psi\|_{C^{1,1}}$ and $\|\psi^{-1}\|_{C^{1,1}}$ (and hence only on $\Omega$)
which may change from appearance to appearance.
Using (\ref{control.area}) and (\ref{simpl.domain.L.infty}) we get
\begin{align*}
\|\phi-\operatorname{id}\|_{L^{\infty}(\Omega)}\leq
C_{\psi}\|\widetilde{\phi}\circ \psi^{-1}-\psi^{-1}\|_{L^{\infty}(\Omega)}\leq
C_{\psi}|M|^{1/2}\tau,
\end{align*}
which proves (\ref{main.L.infty}).
Next note that
\begin{align*}
\nabla \phi-\operatorname{Id}
=\nabla (\psi \circ \widetilde{\phi} \circ \psi^{-1})-\operatorname{Id}
=&\nabla \psi (\widetilde{\phi} \circ \psi^{-1})
\cdot (\nabla \widetilde{\phi} (\psi^{-1})-\operatorname{Id})
\cdot \nabla \psi^{-1}
\\&
+(\nabla \psi (\widetilde{\phi} \circ \psi^{-1})-\nabla \psi (\psi^{-1}))
\cdot \nabla \psi^{-1}
\end{align*}
which yields, using (\ref{control.area}), (\ref{simpl.domain.W.1.p}), and
(\ref{simpl.domain.L.infty}),
\begin{align*}
\|\nabla \phi-\operatorname{Id}\|_{L^p(\Omega)}&
\leq C_{\psi}\|\nabla \widetilde{\phi}-\operatorname{Id}
\|_{L^p([0,1]^2)}
+C_{\psi}\| \widetilde{\phi}-\operatorname{id}
\|_{L^\infty([0,1]^2)}\\
&\leq C_{\psi}|M|^{1/(2p)}\tau,
\end{align*}
proving (\ref{main.W.1.p}).
Next, using (\ref{simpl.domain.nabla.borne.par.det}), we deduce, a.\,e.\ in $\Omega,$
\begin{align*}
|\nabla \phi|&\leq C_{\psi}|\nabla\widetilde{\phi}(\psi^{-1})|\leq C_{\psi}\det\nabla\widetilde{\phi}(\psi^{-1})\leq C_{\psi}\det\nabla\phi,
\end{align*}
which proves (\ref{main.nabla.phi.bounded.by.det}).
Then
\begin{align*}
&\det \nabla (\psi \circ \widetilde{\phi} \circ \psi^{-1})
=\det \nabla \psi (\widetilde{\phi} \circ \psi^{-1})
\det \nabla \widetilde{\phi} (\psi^{-1})
\det\nabla \psi^{-1}
\\
=&\left(\det \nabla \psi (\widetilde{\phi} \circ \psi^{-1})
-\det \nabla \psi (\psi^{-1})
\right)
\det \nabla \widetilde{\phi} (\psi^{-1})
\det\nabla \psi^{-1}
+\det \nabla \widetilde{\phi} (\psi^{-1})
\\
\geq& -C_\psi |\widetilde{\phi}\circ \psi^{-1}-\psi^{-1}|
+\det \nabla \widetilde{\phi} (\psi^{-1})
\\
\geq& -C_\psi |M|^{1/2} \tau
+\det \nabla \widetilde{\phi} (\psi^{-1}),
\end{align*}
where we have used (\ref{simpl.domain.L.infty}) for
the last inequality.
Hence using (\ref{simpl.domain.det.on.M}) and (\ref{simpl.domain.det.global})
we get from the last inequality
\begin{align*}
\det\nabla \phi\geq
\left\{\begin{array}{cl}1+\tau(2-C_{\psi}|M|^{1/2})&\text{a.\,e.\ in $M$}
\\
1-C_{\psi}|M|^{1/2}\tau&\text{a.\,e.\ in $\Omega\setminus M.$}\end{array} \right.
\end{align*}
Taking $|M|$ small enough so that $C_{\psi}|M|^{1/2}\leq 1$
gives (\ref{main.det.on.M}) and (\ref{main.det.global}). \smallskip

\textit{Step 2.} We now prove the lemma in the general case. \smallskip

\textit{Step 2.1.} Let $\tau>0$ and $M\subset \Omega$ with
$|M|$ and $\sqrt{|M|}\tau$ small enough.
Using Lemma \ref{lemma:decomposition.domain}
there
exist $N$ open sets $\Omega_1,\ldots,\Omega_N$ every one of which is
$C^{1,1}$-equivalent to the unit square and such that
\begin{align*}
\overline{\Omega}=\bigcup_{i=1}^{N}\overline{\Omega}_i\quad \text{and}\quad \Omega_i\cap \Omega_j=\emptyset\quad \text{for every $1\leq i<j\leq N$}.
\end{align*}
Using Step 1, for any $1\leq i\leq N$ there
exists a bi-Lipschitz mapping $\phi_{i}$ from $\overline{\Omega}_{i}$ to
$\overline{\Omega}_{i}$  satisfying
(\ref{main.boundary})-(\ref{main.det.global}) for
\begin{align*}
M_i:=M\cap \Omega_i.
\end{align*}
We extend the $\phi_{i}$ by the identity outside $\Omega_i$ which obviously
implies that (\ref{main.W.1.p})-(\ref{main.det.global}) are satisfied on
$\Omega$ (and not only on $\Omega_{i})$. Summarizing we have for every $1\leq
i\leq N$
\begin{align}
\phi_{i}&=\operatorname{id}\quad \text{on }\overline{\Omega}\setminus \Omega_{i}
\quad\text{and}\quad \phi_i(\Omega_i)=\Omega_i,
\label{lemma.id.ouside.Omega.i}
\\
\|\phi_{i}-\operatorname{id}\|_{W^{1,p}(\Omega)}&\leq C|M_{i}|^{1/(2p)}
\tau
\quad \text{for every $1\leq p\leq \infty,$}\label{lemma.W.1.p}
\\
\|\phi_{i}-\operatorname{id}\|_{L^{\infty}(\Omega)}&\leq C|M_{i}|^{1/2}\tau,\label{lemma.L.infty}
\\
|\nabla\phi_{i}|&\leq C\det\nabla \phi_i\quad \text{a.\,e.\ in $\Omega$},\label{lemma.nabla.bounded.by.det}\\
\det \nabla \phi_{i}&\geq 1+\tau\quad \text{a.\,e.\ in $M_{i},$}\label{lemma.det.on.M.i}
\\
\det \nabla \phi_{i}&\geq 1-C|M_{i}|^{1/2}\tau\quad \text{a.\,e.\ in $\Omega.$}
\label{lemma.det.global}
\end{align}

\textit{Step 2.2 (conclusion).}  We claim that
\begin{align*}
\phi:=\phi_{N}\circ\ldots\circ\phi_1
\end{align*}
satisfies all the properties stated in Theorem \ref{BiLipschitzMapTheorem}.
First we obviously have that $\phi$ is a bi-Lipschitz mapping from
$\overline{\Omega}$ to $\overline{\Omega}$
with $\phi=\operatorname{id}$ on $\partial\Omega.$
For every $1\leq i\leq N$, we have
\begin{align*}
\phi=\phi_i\quad \text{in }\Omega_i
\end{align*}
and hence for every $1\leq i\leq N$
\begin{align*}
\nabla \phi=\nabla\phi_i\quad \text{and}\quad \det\nabla \phi=\det\nabla\phi_i\quad \text{a.\,e.\ in $\Omega_i$.}
\end{align*}
Therefore, since that $|M_i|\leq |M|$ and since $|\Omega\setminus \cup_{i=1}^N\Omega_i|=0$ we get
directly get (\ref{main.W.1.p})-(\ref{main.det.global}) from (\ref{lemma.W.1.p})-(\ref{lemma.det.global}).
This ends the proof.
\end{proof}

\subsection{Simplification of the set $M$ on which we stretch}
\label{sub.sect.Simplification.of.the.set.M}

\begin{lemma}\label{Lemma.simpl.set}
In order to prove Theorem \ref{BiLipschitzMapTheorem} for some domain $\Omega\subset\mathbb{R}^2$, it is enough to prove it for compact subsets $M=K\subset \Omega.$
\end{lemma}

\begin{proof} The proof utilizes the inner regularity of the Lebesgue measure,
the weak lower semicontinuity of the norms $\|\cdot\|_{W^{1,p}}$, and the weak continuity of the determinant.\smallskip

Let $\tau>0$ and $M\subset \Omega$ be a measurable set with $|M|$ and $\sqrt{|M|}\tau$ small enough.
We assume that Theorem \ref{BiLipschitzMapTheorem} has been proven in the case when the set $M$ that is to be stretched is a compact subset $K\subset\Omega$. By a limiting process, we show that Theorem \ref{BiLipschitzMapTheorem} then holds for any measurable set $M\subset\Omega$.
Recall that the constant $C$ appearing in (\ref{main.W.1.p}), (\ref{main.L.infty}), \eqref{main.nabla.phi.bounded.by.det}, and (\ref{main.det.global}) is independent of $K$ (and of $\tau$).\smallskip

\textit{Step 1.} By inner regularity of the Lebesgue measure, we may
choose an increasing sequence of compact sets $K_{\nu}\subset M$ with
$\lim_{\nu\rightarrow \infty}|K_{\nu}|=|M|$.
For any $K_{\nu}$ we have by assumption a bi-Lipschitz map
$\phi_{\nu}=(\phi_{\nu})_{\tau,K_\nu}:\overline{\Omega}\rightarrow
\overline{\Omega}$ satisfying
\begin{align}
\phi_{\nu}&=\operatorname{id}\quad  \text{on $\partial\Omega,$}
\label{lemma.set.boundary}
\\
\|\phi_{\nu}-\operatorname{id}\|_{W^{1,p}(\Omega)}&\leq C|K_{\nu}|^{1/(2p)}
\tau\quad
\text{for every $1\leq p\leq \infty,$}
\label{lemma.set.W.1.p}
\\
\|\phi-\operatorname{id}\|_{L^{\infty}(\Omega)}&\leq C|K_{\nu}|^{1/2}\tau,\label{lemma.set.L.infty}\\
|\nabla \phi_{\nu}|&\leq C\det\nabla\phi_{\nu}\quad
\text{a.\,e.\ in $\Omega$,}
\label{lemma.set.nabla.bounded.by.det.W.1.p}\\
\det \nabla \phi_\nu&\geq 1+\tau\quad
\text{a.\,e.\ in $K_{\nu},$}\label{lemma.set.det.on.M}
\\
\det \nabla \phi_{\nu}&\geq 1-C|K_{\nu}|^{1/2}\tau\quad
\text{a.\,e.\ in $\Omega.$}\label{lemma.set.det.global}
\end{align}
Taking the upper bound $c$ on $\sqrt{|M|}\tau$ in the assumptions of Theorem \ref{BiLipschitzMapTheorem} smaller if necessary, we deduce from
(\ref{lemma.set.det.global}) that
\begin{align*}
\det\nabla \phi_{\nu}\geq 1/2\quad \text{a.\,e.\ in $\Omega.$}
\end{align*}

\textit{Step 2.}
By (\ref{lemma.set.W.1.p}) with $p=\infty$ and
the last inequality, the maps $\phi_{\nu},\phi_{\nu}^{-1}$ are
uniformly bounded in $W^{1,\infty}(\Omega),$ namely
\begin{equation}
\label{conv.W.1.infty.phi.n}
\|\phi_{\nu}-\operatorname{id}\|_{W^{1,\infty}(\Omega)}
\leq C\tau
\end{equation}
and
\begin{align*}
\sup_{\nu}\|\phi_{\nu}^{-1}\|_{W^{1,\infty}}<\infty.
\end{align*}
Hence, up to a subsequence we know that
\begin{align*}
\phi_{\nu}\stackrel{\ast}{\rightharpoonup}\phi\quad
\text{in $W^{1,\infty}(\Omega)$ as $\nu\rightarrow \infty$}
\end{align*}
holds for some bi-Lipschitz map $\phi$ from $\overline{\Omega}$ to
$\overline{\Omega}.$
By compactness of the embedding of $W^{1,\infty}$ in $C^0$,
we see that $\phi_{\nu}$ converges to $\phi$ also in $C^0(\overline{\Omega}).$
Hence, using (\ref{lemma.set.boundary}) we get \eqref{main.boundary}, namely
\begin{align*}
\phi=\operatorname{id}\quad \text{on $\partial \Omega.$}
\end{align*}
Note also that by the weak lower semicontinuity we have
\begin{equation}
\label{weak.lower.semi.cont}
\|\phi-\operatorname{id}\|_{W^{1,p}(\Omega)}
\leq \liminf_{\nu\rightarrow \infty}
\|\phi_{\nu}-\operatorname{id}\|_{W^{1,p}(\Omega)}\quad
\text{for every $1\leq p\leq \infty.$}
\end{equation}
Hence, from (\ref{lemma.set.W.1.p}), (\ref{weak.lower.semi.cont}), and the fact
that $|K_{\nu}|\leq |M|$ holds we get \eqref{main.W.1.p}, namely
\begin{align*}
\|\phi-\operatorname{id}\|_{W^{1,p}(\Omega)}\leq C|M|^{1/(2p)}\tau \quad \text{for every $1\leq p\leq \infty.$}
\end{align*}
Furthermore, from
\begin{align*}
\|\phi_\nu-\operatorname{id}\|_{L^\infty(\Omega)}
\leq C\sqrt{|K_\nu|}\tau\leq C\sqrt{|M|}\tau
\end{align*}
we deduce \eqref{main.L.infty}, namely
\begin{align*}
\|\phi-\operatorname{id}\|_{L^\infty(\Omega)}
\leq C\sqrt{|M|}\tau.
\end{align*}

\textit{Step 3.} It remains to prove the three assertions involving the
determinant, namely (\ref{main.nabla.phi.bounded.by.det})-(\ref{main.det.global}). By weak continuity of the determinant, we have for any $f\in
L^{\infty}(\Omega)$
\begin{equation}\label{weak.cont.det}
\lim_{\nu\rightarrow \infty} \int_{\Omega} f\det \nabla \phi_{\nu}~dx
=\int_{\Omega} f\det \nabla \phi~dx.
\end{equation}
By sequential lower semicontinuity of the norm $||\nabla \phi_\nu ||_{L^1(A)}$ for any open set $A\subset\Omega$, we deduce from (\ref{lemma.set.nabla.bounded.by.det.W.1.p}) and (\ref{weak.cont.det}) that
\begin{align*}
\int_{\Omega}|\nabla \phi|\chi_{A} ~dx
&\leq \liminf_{\nu\rightarrow \infty} ||\nabla \phi_\nu||_{L^1(A)}
=\liminf_{\nu\rightarrow \infty}\int_{\Omega}|\nabla \phi_{\nu}|\chi_{A} ~dx
\\
 &\leq C\lim_{\nu\rightarrow \infty}\int_{\Omega}\det\nabla \phi_{\nu}\chi_{A} ~dx
 =C\int_{\Omega}\det\nabla \phi\chi_{A} ~dx,
\end{align*}
which proves (\ref{main.nabla.phi.bounded.by.det}) since the open set $A$ is arbitrary.
Recalling that (see (\ref{lemma.set.det.on.M})) $\det \nabla \phi_{\nu}\geq 1+\tau$
a.\,e.\ on $K_{\nu}$ and that $K_\nu\subset K_{\nu+1}$, we obviously have for any
measurable subset $A\subset K_{\nu}$
\begin{align*}
\int_{\Omega} \chi_A \det\nabla \phi_{\nu+k}~dx
\geq \int_{\Omega} (1+\tau)\chi_A~dx
\quad \forall k\geq 0,
\end{align*}
which implies (by (\ref{weak.cont.det}))
\begin{align*}
\int_{\Omega} \chi_A \det\nabla \phi~dx
\geq \int_{\Omega} (1+\tau)\chi_A~dx.
\end{align*}
By arbitrariness of $A\subset K_\nu$, we obtain $\det \nabla \phi\geq 1+\tau$
a.\,e.\ on $K_\nu$ for all $\nu$. Since $|M\setminus (\cup_{\nu\geq 1}K_{\nu})|=0,$
we deduce
\begin{align*}
\det\nabla \phi\geq 1+\tau \quad \text{a.\,e.\ in $M.$}
\end{align*}
Similarly, using (\ref{lemma.set.det.global}) and
(\ref{weak.cont.det}) we get for any measurable subset $A\subset \Omega$
\begin{align*}
&\int_{\Omega} \chi_A \det \nabla \phi~dx
=\lim_{\nu\rightarrow \infty}\int_{\Omega} \chi_A \det \nabla \phi_{\nu}~dx
\\
\geq&
\lim_{\nu\rightarrow \infty}
\int_{\Omega} (1-C\sqrt{|K_{\nu}|}\tau)\chi_A~dx
= \int_{\Omega} (1-C\sqrt{|M|}\tau)\chi_A~dx,
\end{align*}
which, again by arbitrariness of $A,$ implies
\begin{align*}
\det \nabla \phi
\geq 1-C\sqrt{|M|}\tau\quad \text{a.\,e.\ in $\Omega.$}
\end{align*}
This ends the proof.
\end{proof}

\subsection{Simplification of the boundary values}\label{sub.sect.Simplification.of.the.boundary.values}

We finally notice that it is enough to prove Theorem \ref{BiLipschitzMapTheorem}
when $\Omega=(0,1)^2$, when $M\subset \Omega$ is compact and without assuming the
bi-Lipschitz mapping to preserve the boundary pointwise.

\begin{lemma}\label{lemma.simpl.boundary.values}
Theorem \ref{BiLipschitzMapTheorem} is implied by Proposition \ref{Prop.Bi.Lip.simpl} below.
\end{lemma}

\begin{proposition}\label{Prop.Bi.Lip.simpl}
There exist universal constants $C,c>0$  with the following property: for any
$\tau>0$ and any compact set $K\subset (0,1)^2$ with
$\max\{|K|,\sqrt{|K|}\tau\}\leq c$ there exists a bi-Lipschitz mapping
$\phi=\phi_{\tau,K}:[0,1]^2\rightarrow [0,1]^2$ satisfying
\begin{align}
\|\phi-\operatorname{id}\|_{W^{1,p}([0,1]^2)}&\leq C|K|^{1/(2p)}\tau
\quad \text{for all $1\leq p\leq \infty,$}
\label{prop.W.1.p}
\\
\|\phi-\operatorname{id}\|_{L^{\infty}([0,1]^2)}&\leq C\sqrt{|K|}\tau
\label{prop.L.infty}
\\
|\nabla\phi|&\leq C\det\nabla\phi\quad \text{a.\,e.\ in $[0,1]^2,$}
\label{prop.nabla.phi.bounded.by.det}
\\
\det \nabla \phi&\geq 1+\tau\quad \text{a.\,e.\ on $K,$}
\label{prop.det.on.M}
\\
\det \nabla \phi&\geq 1-C\sqrt{|K|}\tau\quad \text{a.\,e.\ on $[0,1]^2.$}
\label{prop.det.global}
\end{align}
Moreover the following properties concerning  the boundary values hold true:
\begin{align}
\phi_1(0,s)&=\phi_2(s,0)=0\quad\text{for every $s\in [0,1],$}\label{prop.bord.pres.glob.1}
\\
\phi_1(1,s)&=\phi_2(s,1)=1\quad\text{for every $s\in [0,1],$}\label{prop.bord.pres.glob.2}
\\
\partial_x \phi_1(x,y) &\geq 1-C|K|^{1/2}\tau\quad \text{for every }(x,y)\in [0,1]\times \{0,1\},
\label{prop.derivees.phi.1}
\\
\partial_y \phi_2(x,y) &\geq 1-C|K|^{1/2}\tau  \quad \text{for every }(x,y)\in \{0,1\}\times [0,1]
\label{prop.derivees.phi.2}.
\end{align}
\end{proposition}

\begin{proof}[Proof of Lemma \ref{lemma.simpl.boundary.values}.]

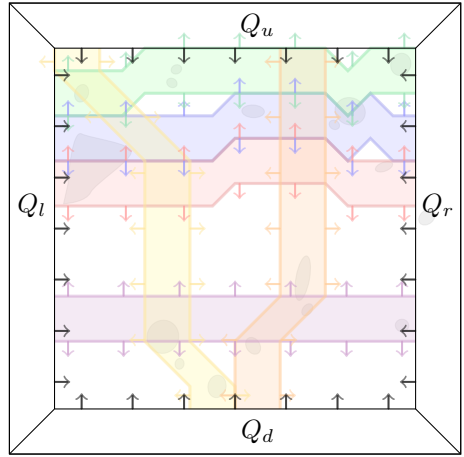
\begin{figure}
\begin{tikzpicture}[scale=0.6]
\draw[rotate around={-20:(0.1,6.5)},fill,gray,opacity=0.3] (0.1,6.5) ellipse
(0.15cm and 0.2cm);
\draw[rotate around={-30:(2.1,6.73)},fill,gray,opacity=0.3] (2.1,6.73) ellipse
(0.12cm and 0.1cm);
\draw[rotate around={25:(2.2,7.03)},fill,gray,opacity=0.3] (2.2,7.03) ellipse
(0.12cm and 0.1cm);
\draw[fill,gray,opacity=0.3] (-0.3,4.05) .. controls (-0.2,4.05) and (0.2,4.1) ..
(0.3,4.3) .. controls (0.4,4.5) and (1.2,5.0) .. (1.2,5.2)
.. controls (1.2,5.4) and (0.3,5.5) .. (0.1,5.6) .. controls (-0.1,5.7) and
(-0.2,4.7) .. (-0.2,4.5) -- cycle;
\draw[rotate around={80:(3.1,0.0)},fill,gray,opacity=0.3] (3.1,0.0) ellipse
(0.25cm and 0.2cm);
\draw[rotate around={-20:(1.9,1.1)},fill,gray,opacity=0.3] (1.9,1.1) ellipse
(0.35cm and 0.37cm);
\draw[rotate around={81:(2.27,0.5)},fill,gray,opacity=0.3] (2.27,0.5) ellipse
(0.12cm and 0.1cm);
\draw[rotate around={-10:(5,2.4)},fill,gray,opacity=0.3] (5,2.4) ellipse
(0.15cm and 0.5cm);
\draw[rotate around={51:(5.1,1.7)},fill,gray,opacity=0.3] (5.1,1.7) ellipse
(0.15cm and 0.1cm);
\draw[rotate around={-10:(7.2,1.35)},fill,gray,opacity=0.3] (7.2,1.35) ellipse
(0.15cm and 0.14cm);
\draw[rotate around={-70:(6.8,4.87)},fill,gray,opacity=0.3] (6.8,4.87) ellipse
(0.1cm and 0.2cm);
\draw[rotate around={40:(7.75,3.75)},fill,gray,opacity=0.3] (7.75,3.75) ellipse
(0.2cm and 0.15cm);
\draw[rotate around={-20:(7.3,7.2)},fill,gray,opacity=0.3] (7.2,7.1) ellipse
(0.22cm and 0.25cm);
\draw[rotate around={-20:(6.05,6.1)},fill,gray,opacity=0.3] (6.05,6.1) ellipse
(0.34cm and 0.3cm);
\draw[rotate around={-30:(5.7,5.8)},fill,gray,opacity=0.3] (5.7,5.8) ellipse
(0.12cm and 0.1cm);
\draw[rotate around={33:(3.9,0.9)},fill,gray,opacity=0.3] (3.9,0.9) ellipse
(0.15cm and 0.2cm);
\draw[rotate around={87:(3.9,6.1)},fill,gray,opacity=0.3] (3.9,6.1) ellipse
(0.14cm and 0.25cm);
\draw[fill,green,opacity=0.2] (-0.5,7) -- (0,7) -- (1,7) -- (1.5,7.5)
-- (5.5,7.5) -- (6,7) -- (6.5,7.5) -- (7.5,7.5)
-- (7.5,6.5) -- (6.5,6.5) -- (6,6) -- (5.5,6.5) -- (1.5,6.5) -- (1,6)
-- (0,6) -- (-0.5,6);
\draw[fill,blue,opacity=0.2] (-0.5,6) -- (0,6) -- (3,6) -- (3.5,6.5)
-- (5.5,6.5) -- (6,6) -- (6.5,6.5) -- (7,6) -- (7.5,6)
--(7.5,5) -- (7,5) -- (6.5,5.5) -- (6,5) -- (5.5,5.5) -- (3.5,5.5) -- (3,5)
-- (0,5) -- (-0.5,5);
\draw[fill,red,opacity=0.17] (-0.5,5) -- (0,5) -- (3,5) -- (3.5,5.5)
-- (5.5,5.5) -- (6,5) -- (7,5) -- (7.5,5)
-- (7.5,4) -- (7,4) -- (6,4) -- (5.5,4.5) -- (3.5,4.5) -- (3,4) -- (0,4)
-- (-0.5,4);
\draw[fill,violet,opacity=0.2] (-0.5,1) -- (7,1) -- (7.5,1)
-- (7.5,2) -- (7,2) -- (-0.5,2);
\draw[fill,yellow,opacity=0.27] (0.5,7.5) -- (0.5,7) -- (2.5,5) -- (2.5,1)
-- (3.5,0) -- (3.5,-0.5)
-- (2.5,-0.5) -- (2.5,0) -- (1.5,1) -- (1.5,5) -- (-0.5,7) -- (-0.5,7.5);
\draw[fill,orange,opacity=0.24] (5.5,7.5) -- (5.5,3) -- (5.5,2) -- (4.5,1) --
(4.5,-0.5)
-- (3.5,-0.5) -- (3.5,1) -- (4.5,2) -- (4.5,3) -- (4.5,7.5);
\draw (-1.5,-1.5) -- (-1.5,8.5) -- (8.5,8.5) -- (8.5,-1.5) -- cycle;
\draw[blue!80!black,very thick,opacity=0.4] (-0.5,6) -- (0,6) -- (3,6) -- (3.5,6.5)
-- (5.5,6.5) -- (6,6) -- (6.5,6.5) -- (7,6) -- (7.5,6);
\draw[blue!80!black,very thick,opacity=0.4] (7.5,5) -- (7,5) -- (6.5,5.5)
-- (6,5) -- (5.5,5.5) -- (3.5,5.5) -- (3,5) -- (0,5) -- (-0.5,5);
\draw[red!75!black,very thick,opacity=0.4] (-0.5,5) -- (0,5) -- (3,5) -- (3.5,5.5)
-- (5.5,5.5) -- (6,5) -- (7,5) -- (7.5,5);
\draw[red!75!black,very thick,opacity=0.4]
(7.5,4) -- (7,4) -- (6,4) -- (5.5,4.5) -- (3.5,4.5) -- (3,4) -- (0,4)
-- (-0.5,4);
\draw[green!75!blue,very thick,opacity=0.5]
(-0.5,7) -- (0,7) -- (1,7) -- (1.5,7.5) -- (5.5,7.5) -- (6,7) -- (6.5,7.5)
-- (7.5,7.5);
\draw[green!75!blue,very thick,opacity=0.5]
(7.5,6.5) -- (6.5,6.5) -- (6,6) -- (5.5,6.5) -- (1.5,6.5) -- (1,6)
-- (0,6) -- (-0.5,6);
\draw[violet,very thick,opacity=0.4] (-0.5,1) -- (7,1) -- (7.5,1);
\draw[violet,very thick,opacity=0.4] (7.5,2) -- (7,2) -- (-0.5,2);
\draw[yellow!75!red,very thick,opacity=0.5] (0.5,7.5) -- (0.5,7) -- (2.5,5)
-- (2.5,1) -- (3.5,0) -- (3.5,-0.5);
\draw[yellow!75!red,very thick,opacity=0.5]
(2.5,-0.5) -- (2.5,0) -- (1.5,1) -- (1.5,5) -- (-0.5,7) -- (-0.5,7.5);
\draw[orange,very thick,opacity=0.5] (5.5,7.5) -- (5.5,3) -- (5.5,2) -- (4.5,1) --
(4.5,-0.5);
\draw[orange,very thick,opacity=0.5]
(3.5,-0.5) -- (3.5,1) -- (4.5,2) -- (4.5,3) -- (4.5,7.5);
\foreach \x/\y in {-0.2/5,1.08/5,2.37/5,3.6/5.5,4.83/5.5,6.1/5,7.25/5}
  \draw[red,thick,->,opacity=0.5] (\x,\y) -- (\x,\y+0.35);
\foreach \x/\y in {-0.2/4,1.08/4,2.37/4,3.6/4.5,4.83/4.5,6.1/4,7.25/4}
  \draw[red,thick,->,opacity=0.5] (\x,\y) -- (\x,\y-0.35);
\foreach \x/\y in {-0.2/6,1.08/6,2.37/6,3.6/6.5,4.83/6.5,6.1/6,7.25/6}
  \draw[blue,thick,->,opacity=0.5] (\x,\y) -- (\x,\y+0.35);
\foreach \x/\y in {-0.2/5,1.08/5,2.37/5,3.6/5.5,4.83/5.5,6.1/5,7.25/5}
  \draw[blue,thick,->,opacity=0.5] (\x,\y) -- (\x,\y-0.35);
\foreach \x/\y in {-0.2/7,1.08/7.08,2.37/7.5,3.6/7.5,4.83/7.5,6.1/7.1,7.25/7.5}
  \draw[green!70!blue,thick,->,opacity=0.5] (\x,\y) -- (\x,\y+0.35);
\foreach \x/\y in {-0.2/6,1.08/6.08,2.37/6.5,3.6/6.5,4.83/6.5,6.1/6.1,7.25/6.5}
  \draw[green!70!blue,thick,->,opacity=0.5] (\x,\y) -- (\x,\y-0.35);
\foreach \x/\y in {-0.2/2,1.03/2,2.27/2,3.5/2,4.73/2,5.97/2,7.2/2}
  \draw[violet,thick,->,opacity=0.5] (\x,\y) -- (\x,\y+0.35);
\foreach \x/\y in {-0.2/1,1.03/1,2.27/1,3.5/1,4.73/1,5.97/1,7.2/1}
  \draw[violet,thick,->,opacity=0.5] (\x,\y) -- (\x,\y-0.35);
\foreach \x/\y in {3.5/-0.2,2.5/1.03,2.5/2.27,2.5/3.5,2.5/4.73,1.51/5.97,0.5/7.2}
  \draw[yellow!75!red,thick,->,opacity=0.5] (\x,\y) -- (\x+0.35,\y);
\foreach \x/\y in {2.5/-0.2,1.5/1.03,1.5/2.27,1.5/3.5,1.5/4.73,0.51/5.97,-0.5/7.2}
  \draw[yellow!75!red,thick,->,opacity=0.5] (\x,\y) -- (\x-0.35,\y);
\foreach \x/\y in {3.5/-0.2,3.53/1.03,4.5/2.27,4.5/3.5,4.5/4.73,4.5/5.97,4.5/7.2}
  \draw[orange,thick,->,opacity=0.5] (\x,\y) -- (\x-0.35,\y);
\foreach \x/\y in {4.5/-0.2,4.53/1.03,5.5/2.27,5.5/3.5,5.5/4.73,5.5/5.97,5.5/7.2}
  \draw[orange,thick,->,opacity=0.5] (\x,\y) -- (\x+0.35,\y);
\draw[opacity=0.6,fill=white,draw=white] (-1.3,-1.3) -- (-1.3,8.3) -- (8.3,8.3) -- (8.3,-1.3) -- cycle;
\draw (-1,4) node{$Q_l$};
\draw (8,4) node{$Q_r$};
\draw (4,8) node{$Q_u$};
\draw (4,-1) node{$Q_d$};
\draw (-1.5,-1.5) -- (-1.5,8.5) -- (8.5,8.5) -- (8.5,-1.5) -- (-1.5,-1.5);
\draw (-0.5,-0.5) -- (-0.5,7.5) -- (7.5,7.5) -- (7.5,-0.5) -- (-0.5,-0.5);
\draw (-0.5,-0.5) -- (-1.5,-1.5);
\draw (-0.5,7.5) -- (-1.5,8.5);
\draw (7.5,-0.5) -- (8.5,-1.5);
\draw (7.5,7.5) -- (8.5,8.5);
\foreach \x/\y in
{0.1/-0.5,1.23/-0.5,2.37/-0.5,3.5/-0.5,4.63/-0.5,5.77/-0.5,6.9/-0.5}
  \draw[black,thick,->,opacity=0.5] (\x,\y) -- (\x,\y+0.35);
\foreach \x/\y in
{0.1/7.5,1.23/7.5,2.37/7.5,3.5/7.5,4.63/7.5,5.77/7.5,6.9/7.5}
  \draw[black,thick,->,opacity=0.5] (\x,\y) -- (\x,\y-0.35);
\foreach \x/\y in
{7.5/0.1,7.5/1.23,7.5/2.37,7.5/3.5,7.5/4.63,7.5/5.77,7.5/6.9}
  \draw[black,thick,->,opacity=0.5] (\x,\y) -- (\x-0.35,\y);
\foreach \x/\y in
{-0.5/0.1,-0.5/1.23,-0.5/2.37,-0.5/3.5,-0.5/4.63,-0.5/5.77,-0.5/6.9}
  \draw[black,thick,->,opacity=0.5] (\x,\y) -- (\x+0.35,\y);
\end{tikzpicture}
\caption{A sketch of the construction for the correction of the boundary values. In the inner square, the construction from the proof of Proposition~\ref{Prop.Bi.Lip.simpl} is depicted. \label{BoundaryCorrection}}
\end{figure}

First by Lemmas \ref{Lemma.simpl.domain} and \ref{Lemma.simpl.set} we know that it is enough to prove Theorem \ref{BiLipschitzMapTheorem} when $\Omega=(0,1)^2$ (or equivalently when $\Omega=(-1,1)^2$ -- due to certain symmetries in our construction below, it will be convenient to work on $(-1,1)^2$) and when $M$ (the set that will be stretched) is compact.
It hence remains to prove that Proposition \ref{Prop.Bi.Lip.simpl} implies Theorem \ref{BiLipschitzMapTheorem} in the case of $\Omega=(-1,1)^2$ and compact sets $M\subset (-1,1)^2$. We basically have to show how to change the stretching map $\phi$ so that it satisfies $\phi=\operatorname{id}$ on the boundary (while preserving the other properties).
The idea of the proof goes as follows:

We consider the subsquare $(-1+|M|^{1/2},1-|M|^{1/2})^2\subset (-1,1)^2$ (i.e. we allow for a boundary layer of thickness $|M|^{1/2}$) and use Proposition \ref{Prop.Bi.Lip.simpl}  to obtain a stretching map on the subsquare. On the boundary layer, we interpolate between the boundary values of our stretching map on the subsquare and the identity boundary conditions on the original square. Furthermore, we stretch the full boundary layer, which means that we have to compress the subsquare in the interior slightly. Due to the size $|M|^{1/2}$ of the boundary layer, this will not destroy the stretching property on $M$ in the subsquare.
\smallskip

\textit{Step 1.} We consider the subsquare $S:=(-1+|M|^{1/2},1-|M|^{1/2})^2$
and apply (the rescaled version of) Lemma
\ref{Prop.Bi.Lip.simpl} to the set $M\cap S$ with $2\tau$ in place of $\tau$.
This yields a bi-Lipschitz map $\tilde \phi:\overline{S}\rightarrow \overline{S}$ with the properties (where we abbreviate $I_M:=(-1+|M|^{1/2},1-|M|^{1/2})$ and use the fact that we may assume that $|M|^{1/2}\leq \frac{1}{2}$)
\begin{subequations}
\label{PropertiesTildePhi}
\begin{align}
\label{W1pBoundDerivative}
\|\tilde\phi-\operatorname{id}\|_{W^{1,p}(S)}&\leq C|M|^{1/(2p)}
\tau\quad
\text{for every }1\leq p\leq \infty,
\\
\label{LinftyTildePhi}
\|\tilde\phi-\operatorname{id}\|_{L^\infty(S)}&\leq C|M|^{1/2}\tau,
\\
|\nabla \tilde\phi|&\leq C\det \nabla \tilde\phi \quad \text{a.\,e.\ in }S,
\\
\det \nabla \tilde\phi&\geq 1+2\tau\quad \text{a.\,e.\ on }M\cap S,
\\
\det \nabla \tilde\phi&\geq 1-C|M|^{1/2}\tau\quad \text{a.\,e.\ in }S,
\\
\tilde\phi_2(s,-1+|M|^{1/2})&=\tilde\phi_1(-1+|M|^{1/2},s)=-1+|M|^{1/2}\quad
\text{for }s\in I_M,\label{bord.phi.tilde.x}
\\
\tilde\phi_2(s,1-|M|^{1/2})&=\tilde\phi_
1(1-|M|^{1/2},s)=1-|M|^{1/2}\quad
\text{for }s\in I_M,\label{bord.phi.tilde.y}
\\
\partial_x \tilde\phi_1(x,y) &\geq 1-C|M|^{1/2}\tau \quad \text{for every }(x,y)\in [0,1]\times \{0,1\},
\\
\label{LowerBoundyDerivative}
\partial_y \tilde\phi_2(x,y) &\geq 1-C|M|^{1/2}\tau \quad  \text{for every }(x,y)\in \{0,1\} \times [0,1].
\end{align}
\end{subequations}

\textit{Step 2.} We now define our map $\phi:[-1,1]^2\rightarrow [-1,1]^2$; a sketch of our
construction is provided in Figure \ref{BoundaryCorrection}. On the subsquare
$\overline{S}$, we set
\begin{align}
\label{DefinitionPhiCentralSquare}
\phi(x,y):=
\frac{1-|M|^{1/2}(1+2\tau)}{1-|M|^{1/2}}
\tilde \phi(x,y).
\end{align}
It remains to define $\phi$ in $[-1,1]^2\setminus \overline{S}$ which we divide into four quadrilaterals $Q_l,Q_u,Q_d$ and $Q_r$ (see Figure \ref{BoundaryCorrection}).
In the left quadrilateral $Q_l$, we define $\phi$ by setting
\begin{align}\label{def.on.Q.l}
\phi(x,y):=
\begin{pmatrix}
-1+(1+2\tau)(1+x)
\\
\frac{1-|M|^{1/2}+x}{x|M|^{1/2}} y + \frac{1+x}{|M|^{1/2}}
\phi_2 \left(-1+|M|^{1/2},
\frac{-1+|M|^{1/2}}{x}y
\right)
\end{pmatrix}.
\end{align}
In the other three quadrilaterals, we define $\phi$ by an analogous construction.
We claim that the $\phi$ constructed by this procedure has all the desired properties. This will be done in the remaining two steps.\smallskip

\textit{Step 3.} First, we easily see that $\phi=\operatorname{id}$ on $\partial [-1,1]^2$, i.e. that \eqref{main.boundary} holds.
We claim that $\phi$ is Lipschitz in $[-1,1]^2$. Since $\tilde \phi$ is Lipschitz in $\overline{S}$ and thus, by definition, $\phi$ is Lipschitz in $\overline{S}$, $\overline{Q_l}$, $\overline{Q_r}$, $\overline{Q_u}$, and $\overline{Q_d}$,
it is enough to prove that $\phi$ is continuous on
\begin{align*}
\{(\pm(1-\theta),\pm(1-\theta)),\theta\in [0,\sqrt{|M|}]\}\cup \partial S.
\end{align*}
The continuity of $\phi$ on the first above set (i.\,e.\ the four diagonal segments) is a direct consequence of the definition of $\phi$: For example, on the lower left diagonal we have $\phi(-1+\theta,-1+\theta)=(-1+(1+2\tau)\theta,-1+(1+2\tau)\theta)$ for $\theta\in [0,|M|^{1/2}]$. To prove the continuity on $\partial S$, by symmetry of our construction it is enough to prove it on left vertical segment of $S$, i.e. on
\begin{align*}
\{-1+|M|^{1/2}\}\times I_M.
\end{align*}
First, by (\ref{bord.phi.tilde.x}) and (\ref{DefinitionPhiCentralSquare}) we have
\begin{align}\label{equ.for.phi.1}
\phi_1(-(1-|M|^{1/2}),y)=-1+|M|^{1/2}(1+2\tau).
\text{ for }y\in I_M.
\end{align}
Using (\ref{def.on.Q.l}), (\ref{DefinitionPhiCentralSquare}), and (\ref{equ.for.phi.1}), we have, for any $y\in I_M$,
\begin{align*}
\lim_{x\nearrow -1+|M|^{1/2}} \phi(x,y) =
\begin{pmatrix}
-1+(1+2\tau)|M|^{1/2}
\\
\phi_2(-1+|M|^{1/2},y)
\end{pmatrix}
=
\phi(-1+|M|^{1/2},y),
\end{align*}
proving the claim.

\textit{Step 4.}
In this step we verify the assertions \eqref{main.W.1.p} through \eqref{main.det.global}. It is sufficient to check these assertions separately on $S,Q_l,Q_r,Q_u$ and $Q_d$. In the central square $S$, these properties are a straightforward consequence (taking $|M|$ and $\tau |M|^{1/2}$ smaller if necessary) of the properties \eqref{PropertiesTildePhi} and our definition \eqref{DefinitionPhiCentralSquare}, since the prefactor in \eqref{DefinitionPhiCentralSquare} may be rewritten as
\begin{align*}
1-\frac{2\tau |M|^{1/2}}{1-|M|^{1/2}}.
\end{align*}
It is therefore sufficient to check \eqref{main.W.1.p} through \eqref{main.det.global} on the left quadrilateral $Q_l$, as the construction of $\phi$ in the three other quadrilaterals is analogous.
\smallskip

\textit{Step 4.1.} In $Q_l$, we have
\begin{align}
\label{LemmaPhiDx}
\partial_x \phi(x,y)
=
\begin{pmatrix}
1+2\tau
\\
\frac{-1+|M|^{1/2}}{x^2 |M|^{1/2}}y[1-(1+x)\partial_y \phi_2(\ldots)]
+\frac{1}{|M|^{1/2}}\phi_2(\ldots)
\end{pmatrix}
\end{align}
and
\begin{align}\label{LemmaPhiDy}
\partial_y \phi(x,y)
&=
\begin{pmatrix}
0
\\
\frac{1-|M|^{1/2}+x}{x|M|^{1/2}}+\frac{(1+x)(-1+|M|^{1/2})}{x|M|^{1/2}}
\partial_y\phi_2(\ldots)
\end{pmatrix}
\end{align}
where "$(\ldots)$" stands (and will stand) for
\begin{align*}
\left(-1+|M|^{1/2},
\frac{-1+|M|^{1/2}}{x}y
\right).
\end{align*}
From  (\ref{LowerBoundyDerivative}) and (\ref{DefinitionPhiCentralSquare}) we deduce
\begin{align} \partial_{y}\phi_2(\ldots)\geq 1-C\tau |M|^{1/2}\label{phi.2.y.lower.bound}.
\end{align}
From \eqref{LemmaPhiDx}, \eqref{LemmaPhiDy} and \eqref{phi.2.y.lower.bound} we deduce that, for $(x,y)\in Q_l,$
\begin{align}
\nonumber
\det \nabla \phi(x,y)&=(1+2\tau) \left(
\frac{1-|M|^{1/2}+x}{x|M|^{1/2}}
+\frac{(1+x)(-1+|M|^{1/2})}{x|M|^{1/2}}\partial_y\phi_2(\ldots)
\right)\\ \nonumber
&\geq (1+2\tau) \left(
\frac{1-|M|^{1/2}+x}{x|M|^{1/2}}
+\frac{(1+x)(-1+|M|^{1/2})}{x|M|^{1/2}}(1-C\tau |M|^{1/2})
\right)
\\&
\label{LemmaPhiDet}
= (1+2\tau) \left(1-C \tau |M|^{1/2} \frac{(1+x)(-1+|M|^{1/2})}{x|M|^{1/2}} \right)
\geq 1+\tau,
\end{align}
where in the last inequality we have used the fact that we have $|1+x|\leq |M|^{1/2}$ in $Q_l$ (as well as the smallness condition on $|M|$ and $|M|^{1/2} \tau$).
Therefore \eqref{main.det.on.M} and \eqref{main.det.global} are established.
\smallskip

\textit{Step 4.2.}
From \eqref{W1pBoundDerivative} for $p=\infty$, \eqref{LinftyTildePhi}, and \eqref{DefinitionPhiCentralSquare} we deduce (see also the beginning of Step 4) that
 \begin{align}\label{equ.intermediaire}
 |\partial_y\phi_2(\ldots)-1|\leq C\tau\quad \text{and}\quad \left|\phi_2(\ldots)-y\frac{-1+|M|^{1/2}}{x}\right|\leq C\tau |M|^{1/2}.
 \end{align}
By \eqref{LemmaPhiDx} and \eqref{equ.intermediaire} we have for $(x,y)\in Q_l$
\begin{align}
\Big|\partial_x \phi(x,y)-
\begin{pmatrix}
1\\0
\end{pmatrix}
\Big|
\leq& C\tau + \left|\frac{-1+|M|^{1/2}}{x^2 |M|^{1/2}}y(1+x)(\partial_y \phi_2(\ldots)-1)\right|
\nonumber\\&
+\frac{1}{|M|^{1/2}}\left|\phi_2(\ldots)-y\frac{-1+|M|^{1/2}}{x}\right|
\nonumber\\
\leq&
C\tau+C\tau+\frac{1}{|M|^{1/2}} C|M|^{1/2}\tau
\nonumber\\
=&C\tau,\label{estimate.phi.dx}
\end{align}
where in the last inequality we have used the fact that $|1+x|\leq |M|^{1/2}$ holds in the left quadrilateral $Q_l$.

Similarly, we have by \eqref{LemmaPhiDy} and \eqref{equ.intermediaire} for $(x,y)\in Q_l$
\begin{align}
\Big|\partial_y \phi(x,y)-
\begin{pmatrix}
0\\1
\end{pmatrix}
\Big|
\leq \left|
\frac{(1+x)(1-|M|^{1/2})}{-x|M|^{1/2}}
(\partial_y\phi_2(\ldots)-1)\right|
\leq C\tau,\label{estimate.phi.dy}
\end{align}
where in the last inequality we have used the fact that $|1+x|\leq |M|^{1/2}$ holds in the left quadrilateral.\smallskip

\textit{Step 4.3.}
As the area of the left quadrilateral $Q_l$ is bounded by $|M|^{1/2}$,  \eqref{estimate.phi.dx} and \eqref{estimate.phi.dy} establish the bound \eqref{main.W.1.p}. Next from \eqref{LemmaPhiDet},  \eqref{estimate.phi.dx} and \eqref{estimate.phi.dy} we directly get \eqref{main.nabla.phi.bounded.by.det}. It remains to show \eqref{main.L.infty}. This however is an easy consequence of the bound \eqref{main.W.1.p} for $p=\infty$, the fact that on the left edge of our left quadrilateral we have $\phi=\id$, as well as the fact that any point in the left quadrilateral has distance of at most $|M|^{1/2}$ to the left edge of the quadrilateral.\smallskip

\textit{Step 5.}  Since $\phi=\operatorname{id}$ on $\partial \Omega$, since $\phi$ is Lipschitz and since from \eqref{main.det.global}, up to taking $\tau |M|^{1/2}$ smaller if necessary,
\begin{align*}
\det\nabla \phi\geq 1/2,
\end{align*}
classical degree theory results show that $\phi$ is bi-Lipschitz   (see e.\,g.\ Theorem 2 in \cite{Ball2}).
This concludes the proof.
\end{proof}

\subsection{Covering by strips}\label{SubSection.covering}

The following lemma is a consequence of the (combinatorial) Erd\H{o}s-Szekeres theorem, see Theorem 2.1 in Alberti, Cs\"ornyei, and Preiss \cite{Alberti}.
\begin{lemma}
\label{CombinatorialResult}
Let $S\subset \mathbb{R}^2$ be a set containing a finite number of points. Then there exist $N$ functions $f_i: \mathbb{R} \rightarrow \mathbb{R}$ and $M$ functions $g_j: \mathbb{R} \rightarrow \mathbb{R}$ with the following properties:
\begin{itemize}
\item We have the estimate $N\leq \sqrt{\sharp S}$ and $M\leq \sqrt{\sharp S}$.
\item The functions $f_i$ and $g_j$ are Lipschitz continuous with Lipschitz constant of at most $1$.
\item The set $S$ is contained in the union of the graphs of the $f_i$ (considered as functions of $x$) and the graphs of the $g_j$ (considered as functions of $y$), i.e.\ we have
\begin{align*}
S\subset \bigcup_{1 \leq i \leq N} \{(x,f_i(x)):x\in \mathbb{R}\} \cup \bigcup_{1 \leq j \leq M} \{(g_j(y),y):y \in \mathbb{R}\}.
\end{align*}
\end{itemize}
\end{lemma}
\begin{remark}\label{remark.no.generalization.higher.dim}
The previous lemma is no longer true in higher dimensions (see Question 8.2 in
\cite{Alberti}).
\end{remark}

Lemma \ref{CombinatorialResult} has the following consequence, as observed by Alberti, Cs\"ornyei, and Preiss \cite{Alberti}.
\begin{theorem}[Alberti, Cs\"ornyei, Preiss \cite{Alberti}]
\label{CoveringOfCompactSet}
Let $K\subset [0,1]^2$ be a compact set and let $\epsilon>0$. For any $\delta>0$ small enough (depending on $\epsilon$ and $K$), there exist $N$ functions $f_i:[0,1]\rightarrow [0,1]$ as well as $M$ functions $g_j:[0,1]\rightarrow [0,1]$ such that the following holds true:
\begin{itemize}
\item We have the estimates $\delta N\leq \sqrt{|K|}+\epsilon$ and $\delta M\leq \sqrt{|K|}+\epsilon$.
\item The functions $f_i$ and $g_j$ are Lipschitz continuous with Lipschitz constant no larger than $1$.
\item Considering the horizontal strips $H_i:=\{(x,y):x\in [0,1],|y-f_i(x)|\leq \delta\}$ and the vertical strips $V_j:=\{(x,y):y\in [0,1],|x-g_j(y)|\leq \delta\}$, the set $K$ is covered by these strips, i.e.\ we have
\begin{align*}
K\subset \bigcup_{i=1}^N H_i \cup \bigcup_{j=1}^M V_j.
\end{align*}
\end{itemize}
\end{theorem}
For the reader's convenience, we also state the proof of the result.
\begin{proof} \textit{Step 1.}
Fix $\epsilon>0$. We set $\delta:=2^{-l}$ and subdivide the unit square $[0,1]^2$ into $2^l \times 2^l$ squares. By compactness of $K$, for $l\in \mathbb{N}$ large enough the total area of the subsquares that have nonempty intersection with $K$ is bounded by $|K|+\epsilon^2$ (one easily sees that by monotone convergence of the sequence $\chi_{K_l}$ towards $\chi_K$, where $K_l$ denotes the set containing all subsquares of size $2^{-l} \times 2^{-l}$ that have nonempty intersection with $K$).\smallskip

\textit{Step 2.}
Applying Lemma \ref{CombinatorialResult} to the set of centers of the subsquares contained in $K_l$ -- note that by the previous considerations, the number of such squares is bounded by $2^{2l}(|K|+\epsilon^2)$ -- , we obtain $N$ functions $f_i:\mathbb{R}\rightarrow \mathbb{R}$ and $M$ functions $g_j:\mathbb{R}\rightarrow \mathbb{R}$ whose graphs cover the centers of our squares. Furthermore, we have
$N\leq \sqrt{2^{2l}(|K|+\epsilon^2)}$ and $M\leq \sqrt{2^{2l}(|K|+\epsilon^2)}$, which implies the desired estimates $\delta N\leq \sqrt{|K|}+\epsilon$ and $\delta M\leq \sqrt{|K|}+\epsilon$. We restrict the functions $f_i$ and $g_j$ to the interval $[0,1]$ and replace them by $\min\{\max\{f_i,0\},1\}$ and $\min\{\max\{g_j,0\},1\}$. This obviously preserves the $1$-Lipschitz property; furthermore, the centers of the subsquares are still covered by the graphs of the modified $f_i$ and $g_j$.\smallskip

\textit{Step 3.}
It is easy to see that the union of the strips associated with the (modified) $f_i$ and $g_j$ provides a covering of $K$: A strip $H_i$ must cover a full subsquare whenever the graph $f_i$ covers the center of the subsquare, as the $f_i$ are $1$-Lipschitz functions. For the $V_j$ and $g_j$, the analogous assertion holds.
\end{proof}

We now provide a slightly stronger version of Lemma \ref{CoveringOfCompactSet} which states that
\begin{itemize}
\item one may actually choose the interior of the horizontal strips $H_i$ to be mutually disjoint and similarly for the vertical strips $V_j$, and
\item the strips $H_i$ and $V_j$ can be chosen to be contained in the unit square.
\end{itemize}
This fact has first been observed by Marchese \cite{Marchese} and later, independently, by the authors of the present paper in an earlier version of the paper.
\begin{lemma}
\label{IntersectionFreeCovering}
The following slightly strengthened version of Lemma \ref{CoveringOfCompactSet} holds: We may additionally enforce in Lemma \ref{CoveringOfCompactSet} that we have $f_{i+1}\leq f_i-2\delta$ and $g_{j+1}\leq g_j - 2\delta$ or, equivalently,
\begin{align*}
\sum_{i=1}^N \chi_{\{(x,y):x\in [0,1],|y-f_i(x)|< \delta\}}\leq 1
\quad\text{and}\quad
\sum_{j=1}^M \chi_{ \{(x,y):y\in [0,1],|x-g_j(y)|< \delta\}}\leq 1.
\end{align*}
 Furthermore, we may enforce $\delta\leq f_i\leq 1-\delta$ and $\delta\leq g_j\leq 1-\delta$, or, equivalently,
 $$\bigcup_{i=1}^N H_i \cup \bigcup_{j=1}^M V_j\subset [0,1]^2.$$
\end{lemma}

\begin{proof}\textit{Step 1.}
Let $\tilde f_i$, $\tilde g_j$ be the family of functions obtained by applying Lemma~\ref{CoveringOfCompactSet}. Then define $f_1$ as
\begin{align*}
f_1(x):=\min\big\{\max\{\tilde f_1(x),\ldots,\tilde f_N(x),(1+2(N-1))\delta\},1-\delta\big\}.
\end{align*}
Inductively, define for $2\leq i\leq N$
\begin{align*}
f_{i}(x):=\min\big\{\max\big\{{\max}_{i}\{\tilde f_1(x),\ldots,\tilde f_N(x)\},(1+2(N-i))\delta\big\},f_{i-1}-2\delta\big\}
\end{align*}
where $\max_{i}$ denotes the $i$-th largest number of the set (i.e. in particular\ $\max_1 S = \max S$ and, e.\,g.\
$\max_2\{1,3,3\}=3$).  We define the $g_j$ analogously.
We claim that the $f_i$ and $g_j$ have all the wished properties. This will be shown in the remaining
three steps.
\smallskip

\textit{Step 2.} First we trivially have that
$$f_{i+1}\leq f_i-2\delta, \quad \text{and}\quad f_i\leq 1-\delta$$ and similarly for the $g_j.$ Moreover a direct induction shows that $f_i\geq (1+2(N-i))\delta$ which in particular implies that
$$f_i\geq \delta$$ and similarly for $g_j.$
The Lipschitz estimate on the $\tilde f_i$ from Lemma \ref{CoveringOfCompactSet} easily carries over to the $f_i$ (same for $g_j$). It only remains to prove (see Step 3) the covering property of the modified strips, i.e.
\begin{align}\label{K.covered}
K\subset \bigcup_{i=1}^N \{(x,y):x\in [0,1],|y-f_i(x)|\leq \delta\} \cup \bigcup_{j=1}^M \{(x,y):y\in [0,1],|x-g_j(y)|\leq \delta\}.
\end{align}

\textit{Step 3.}
Since by construction (\ref{K.covered}) is fulfilled with $f_i$ and $g_j$ replaced by $\tilde f_i$ and $\tilde g_j$
and since $K\subset [0,1]^2$ it is enough to show that, for every $x\in [0,1],$
\begin{align}
\label{strips.still.covering.K}
[0,1]\cap \bigcup_{i=1}^N
[\tilde f_i(x)-\delta,\tilde f_i(x)+\delta]
\subset\bigcup_{i=1}^N
[f_i-\delta,f_i+\delta]
\end{align}
and similarly for $\tilde g_j$ and $g_j$ to show (\ref{strips.still.covering.K}).
We only prove the assertion for the $f_i$ (the other one being the same) and proceed in three substeps.\smallskip

\textit{Step 3.1.}
Defining the family of functions $\overline f_i$ by
\begin{align*}
\overline f_{i}(x):={\max}_{i}\{\tilde f_1(x),\ldots,\tilde f_N(x)\big\},
\end{align*} we claim that
\begin{align*}\bigcup_{i=1}^N
[\tilde f_i(x)-\delta,\tilde f_i(x)+\delta]
=\bigcup_{i=1}^N
[\overline f_i(x)-\delta,\overline f_i(x)+\delta].
\end{align*}
Indeed this directly follows since the tuple $(\overline f_{1}(x),\cdots,\overline f_{N}(x))$ is just a reordering of the tuple
$(\tilde f_{1}(x),\cdots,\tilde f_{N}(x))$.\smallskip

\textit{Step 3.2.} Define another family $\hat f_i$ by
\begin{align*}
\hat f_{i}(x):=\max\big\{\overline f_{i}(x),(1+2(N-i))\delta\big\}.
\end{align*}
We claim that
\begin{align*}
[0,1]\cap\bigcup_{i=1}^N
[\overline f_i(x)-\delta,\overline f_i(x)+\delta]\subset
\bigcup_{i=1}^N
[\hat f_{i}(x)-\delta,\hat f_{i}(x)+\delta].
\end{align*}
Assume that, for some $i$, $\hat f_i(x)\neq \overline f_i(x)$ (otherwise the claim is trivial), and thus, by definition,
\begin{align*}\hat f_i(x)=(1+2(N-i))\delta> \overline f_i(x).
\end{align*}
If $\hat f_{i+1}(x)\leq \overline f_i(x) $ we deduce, by definition, that
\begin{align*}(1+2(N-i-1))\delta\leq \hat f_{i+1}(x)\leq \overline f_i(x)<\hat f_i(x)=(1+2(N-i))\delta
\end{align*}
which directly implies that
\begin{align*}
[\overline f_i(x)-\delta,\overline f_i(x)+\delta]\subset[\hat f_{i}(x)-\delta,\hat f_{i}(x)+\delta]\cup [\hat f_{i+1}(x)-\delta,\hat f_{i+1}(x)+\delta]
\end{align*}
and we are done.
If $\hat f_{i+1}(x)> \overline f_i(x)$ we deduce, by definition of $\hat f_{i+1}(x)$ and
since $\overline f_i(x)\geq \overline f_{i+1}(x)$, that
\begin{align*}\hat f_{i+1}(x)=(1+2(N-i-1))\delta.
\end{align*}
Hence, proceeding by induction we end up with one of the two following cases:
\begin{itemize}
\item there exists $i<l<N$ such that
\begin{align*}(1+2(N-l-1))\delta\leq \hat f_{l+1}(x)\leq \overline f_{i}(x)\leq \hat f_{l}(x)=(1+2(N-l))\delta,
\end{align*}
and we are done as before.
\item If such a $l$ does not exist we have $\delta=\hat f_{N}(x)>\overline f_{i}(x)$ in which case we trivially have
\begin{align*}
[0,1]\cap[\overline f_i-\delta,\overline f_i+\delta]\subset[\hat f_{N}(x)-\delta,\hat f_{N}(x)+\delta],
\end{align*}
showing the claim.
\end{itemize}
\textit{Step 3.3.} Recalling that the family $f_i$ is defined as \begin{align*}
f_1(x)=\min\big\{\hat f_{1}(x),1-\delta\big\}.
\end{align*}
and, for $2\leq i\leq N,$
\begin{align*}
f_{i}(x)=\min\big\{\hat f_{i}(x),f_{i-1}-2\delta\big\},
\end{align*}
it is enough, in view of Steps 3.1 and 3.2, to prove that
\begin{align*} [0,1]\cap
\bigcup_{i=1}^N[\hat f_{i}(x)-\delta,\hat f_{i}(x)+\delta]\subset
\bigcup_{i=1}^N[f_{i}(x)-\delta,f_{i}(x)+\delta],
\end{align*}
in order to show (\ref{strips.still.covering.K}) and prove the lemma.\smallskip

First we trivially have
\begin{align*}
[0,1]\cap
[\hat f_{1}(x)-\delta,\hat f_{1}(x)+\delta]\subset
[f_{1}(x)-\delta,f_{1}(x)+\delta]
\end{align*}
Finally, suppose that for some $i\geq 2$ we have $f_i(x)\neq \hat f_{i}(x)$ (otherwise the claim is trivial) and thus, by definition,
\begin{align*}
\hat f_{i}(x)> f_{i}(x)=f_{i-1}-2\delta.
\end{align*}
If $\hat f_{i}(x)\in [f_{i}(x),f_{i-1}(x)]$ then the previous equality directly implies that
\begin{align*}[\hat f_{i}(x)-\delta,\hat f_{i}(x)+\delta]\subset
[f_{i}(x)-\delta,f_{i}(x)+\delta]\cup [f_{i-1}(x)-\delta,f_{i-1}(x)+\delta]
\end{align*}
and we are done.
If $\hat f_{i}(x)>f_{i-1}(x)=\min\{\hat f_{i-1}(x),f_{i-2}-2\delta\}$ we deduce that,
since $\hat f_{i-1}(x)\geq \hat f_{i}(x)$,
\begin{align*} f_{i-1}(x)=f_{i-2}-2\delta.
\end{align*}
Hence proceeding by induction we deduce that one of the cases below occur:
\begin{itemize}
\item there exists some $i>l\geq 2$ such that $\hat f_{i}(x)\in [f_{l}(x),f_{l-1}(x)]$ and
$f_{l}(x)=f_{l-1}(x)-2\delta$ and we done as above.
\item If such a $l$ does not exist we have $\hat f_{i}(x)>f_1(x)=1-\delta$ in which case we trivially have
\begin{align*}
[0,1]\cap
[\hat f_{i}(x)-\delta,\hat f_{i}(x)+\delta]\subset
[f_{1}(x)-\delta,f_{1}(x)+\delta],
\end{align*}
proving the claim.
\end{itemize}

\end{proof}

\subsection{The simplified model case}\label{subsection.stretching}

We now prove Proposition \ref{Prop.Bi.Lip.simpl} giving an explicit formula for the stretching map $\phi.$
\begin{figure}
\begin{floatrow}
\begin{tikzpicture}[scale=0.55]
\draw (-1.5,-1.5) -- (-1.5,8.5) -- (8.5,8.5) -- (8.5,-1.5) -- cycle;
\foreach \x in {-0.5,0.5,1.5,2.5,3.5,4.5,5.5,6.5,7.5}
  \draw[lightgray!70!white] (\x,-1.5) -- (\x,8.5);
\foreach \y in {-0.5,0.5,1.5,2.5,3.5,4.5,5.5,6.5,7.5}
  \draw[lightgray!70!white] (-1.5,\y) -- (8.5,\y);
\foreach \x in {-1,0,1,2,3,4,5,6,7,8}
  \foreach \y in {-1,0,1,2,3,4,5,6,7,8}
    \draw[fill,lightgray] (\x,\y) circle (0.04cm);
\draw[rotate around={-20:(0.1,6.5)},fill,gray,opacity=0.2] (0.1,6.5) ellipse
(0.15cm and 0.2cm);
\draw[rotate around={-30:(2.1,6.73)},fill,gray,opacity=0.2] (2.1,6.73) ellipse
(0.25cm and 0.15cm);
\draw[fill,gray,opacity=0.2] (-0.3,3.8) .. controls (-0.2,3.8) and (0.2,4.1) ..
(0.3,4.3) .. controls (0.4,4.5) and (1.2,5.0) .. (1.2,5.2)
.. controls (1.2,5.4) and (0.3,5.5) .. (0.1,5.6) .. controls (-0.1,5.7) and
(-0.2,4.7) .. (-0.2,4.5) -- cycle;
\draw[rotate around={80:(3.1,0.2)},fill,gray,opacity=0.2] (3.1,0.2) ellipse
(0.25cm and 0.2cm);
\draw[rotate around={-20:(1.9,1.1)},fill,gray,opacity=0.2] (1.9,1.1) ellipse
(0.35cm and 0.37cm);
\draw[rotate around={81:(2.27,0.73)},fill,gray,opacity=0.2] (2.27,0.73) ellipse
(0.12cm and 0.1cm);
\draw[rotate around={-10:(5,2.4)},fill,gray,opacity=0.2] (5,2.4) ellipse
(0.15cm and 0.5cm);
\draw[rotate around={51:(5.1,1.7)},fill,gray,opacity=0.2] (5.1,1.7) ellipse
(0.15cm and 0.1cm);
\draw[rotate around={-10:(7.2,1.35)},fill,gray,opacity=0.2] (7.2,1.35) ellipse
(0.15cm and 0.14cm);
\draw[rotate around={-70:(6.8,4.87)},fill,gray,opacity=0.2] (6.8,4.87) ellipse
(0.1cm and 0.2cm);
\draw[rotate around={40:(5.75,3.75)},fill,gray,opacity=0.2] (5.75,3.75) ellipse
(0.2cm and 0.15cm);
\draw[rotate around={-20:(7.3,7.2)},fill,gray,opacity=0.2] (7.2,7.1) ellipse
(0.22cm and 0.25cm);
\draw[rotate around={-20:(6.05,6.1)},fill,gray,opacity=0.2] (6.05,6.1) ellipse
(0.34cm and 0.3cm);
\draw[rotate around={-30:(5.7,5.8)},fill,gray,opacity=0.2] (5.7,5.8) ellipse
(0.12cm and 0.1cm);
\draw[rotate around={33:(3.9,0.9)},fill,gray,opacity=0.2] (3.9,0.9) ellipse
(0.15cm and 0.2cm);
\draw[rotate around={87:(3.9,6.1)},fill,gray,opacity=0.2] (3.9,6.1) ellipse
(0.14cm and 0.25cm);
\draw[fill,black] (0,6) circle (0.07cm);
\draw[fill,black] (2,7) circle (0.07cm);
\draw[fill,black] (7,7) circle (0.07cm);
\draw[fill,black] (4,6) circle (0.07cm);
\draw[fill,black] (0,4) circle (0.07cm);
\draw[fill,black] (0,5) circle (0.07cm);
\draw[fill,black] (6,4) circle (0.07cm);
\draw[fill,black] (6,6) circle (0.07cm);
\draw[fill,black] (7,5) circle (0.07cm);
\draw[fill,black] (1,5) circle (0.07cm);
\draw[fill,black] (7,1) circle (0.07cm);
\draw[fill,black] (4,1) circle (0.07cm);
\draw[fill,black] (5,3) circle (0.07cm);
\draw[fill,black] (5,2) circle (0.07cm);
\draw[fill,black] (2,1) circle (0.07cm);
\draw[fill,black] (0,7) circle (0.07cm);
\draw[fill,black] (3,0) circle (0.07cm);
\end{tikzpicture}
~~~~~
\begin{tikzpicture}[scale=0.55]
\foreach \x in {-0.5,0.5,1.5,2.5,3.5,4.5,5.5,6.5,7.5}
  \draw[lightgray!70!white] (\x,-1.5) -- (\x,8.5);
\foreach \y in {-0.5,0.5,1.5,2.5,3.5,4.5,5.5,6.5,7.5}
  \draw[lightgray!70!white] (-1.5,\y) -- (8.5,\y);
\draw[rotate around={-20:(0.1,6.5)},fill,gray,opacity=0.2] (0.1,6.5) ellipse
(0.15cm and 0.2cm);
\draw[rotate around={-30:(2.1,6.73)},fill,gray,opacity=0.2] (2.1,6.73) ellipse
(0.25cm and 0.15cm);
\draw[fill,gray,opacity=0.2] (-0.3,3.8) .. controls (-0.2,3.8) and (0.2,4.1) ..
(0.3,4.3) .. controls (0.4,4.5) and (1.2,5.0) .. (1.2,5.2)
.. controls (1.2,5.4) and (0.3,5.5) .. (0.1,5.6) .. controls (-0.1,5.7) and
(-0.2,4.7) .. (-0.2,4.5) -- cycle;
\draw[rotate around={80:(3.1,0.2)},fill,gray,opacity=0.2] (3.1,0.2) ellipse
(0.25cm and 0.2cm);
\draw[rotate around={-20:(1.9,1.1)},fill,gray,opacity=0.2] (1.9,1.1) ellipse
(0.35cm and 0.37cm);
\draw[rotate around={81:(2.27,0.73)},fill,gray,opacity=0.2] (2.27,0.73) ellipse
(0.12cm and 0.1cm);
\draw[rotate around={-10:(5,2.4)},fill,gray,opacity=0.2] (5,2.4) ellipse
(0.15cm and 0.5cm);
\draw[rotate around={51:(5.1,1.7)},fill,gray,opacity=0.2] (5.1,1.7) ellipse
(0.15cm and 0.1cm);
\draw[rotate around={-10:(7.2,1.35)},fill,gray,opacity=0.2] (7.2,1.35) ellipse
(0.15cm and 0.14cm);
\draw[rotate around={-70:(6.8,4.87)},fill,gray,opacity=0.2] (6.8,4.87) ellipse
(0.1cm and 0.2cm);
\draw[rotate around={40:(5.75,3.75)},fill,gray,opacity=0.2] (5.75,3.75) ellipse
(0.2cm and 0.15cm);
\draw[rotate around={-20:(7.3,7.2)},fill,gray,opacity=0.2] (7.2,7.1) ellipse
(0.22cm and 0.25cm);
\draw[rotate around={-20:(6.05,6.1)},fill,gray,opacity=0.2] (6.05,6.1) ellipse
(0.34cm and 0.3cm);
\draw[rotate around={-30:(5.7,5.8)},fill,gray,opacity=0.2] (5.7,5.8) ellipse
(0.12cm and 0.1cm);
\draw[rotate around={33:(3.9,0.9)},fill,gray,opacity=0.2] (3.9,0.9) ellipse
(0.15cm and 0.2cm);
\draw[rotate around={87:(3.9,6.1)},fill,gray,opacity=0.2] (3.9,6.1) ellipse
(0.14cm and 0.25cm);
\draw (-1.5,-1.5) -- (-1.5,8.5) -- (8.5,8.5) -- (8.5,-1.5) -- cycle;
\foreach \x in {-1,0,1,2,3,4,5,6,7,8}
  \foreach \y in {-1,0,1,2,3,4,5,6,7,8}
    \draw[fill,lightgray] (\x,\y) circle (0.04cm);
\draw[very thick,draw=red!80!black] (5,8.5) -- (5,3) -- (5,2) -- (5,1) --
(5,-1.5);
\draw[very thick,draw=red!80!black] (7,8.5) -- (7,7) -- (6,6) -- (7,3) -- (7,2) -- (7,1) --
(7,-1.5);
\draw[very thick,draw=red!80!black] (0,8.5) -- (0,6) -- (1,5) -- (2,1) -- (2,-1.5);
\draw[very thick,draw=blue!80!black] (-1.5,3.5) -- (0,5) -- (0,5) -- (2,7) -- (4,6) -- (6,5) -- (7,5) -- (8.5,5);
\draw[very thick,draw=blue!80!black] (-1.5,4) -- (0,4) -- (6,4) -- (8.5,4);
\draw[very thick,draw=blue!80!black] (-1.5,0) -- (3,0) -- (4,1) -- (8.5,1);
\draw[fill,black] (0,6) circle (0.07cm);
\draw[fill,black] (2,7) circle (0.07cm);
\draw[fill,black] (7,7) circle (0.07cm);
\draw[fill,black] (4,6) circle (0.07cm);
\draw[fill,black] (0,4) circle (0.07cm);
\draw[fill,black] (0,5) circle (0.07cm);
\draw[fill,black] (6,4) circle (0.07cm);
\draw[fill,black] (6,6) circle (0.07cm);
\draw[fill,black] (7,5) circle (0.07cm);
\draw[fill,black] (1,5) circle (0.07cm);
\draw[fill,black] (7,1) circle (0.07cm);
\draw[fill,black] (4,1) circle (0.07cm);
\draw[fill,black] (5,3) circle (0.07cm);
\draw[fill,black] (5,2) circle (0.07cm);
\draw[fill,black] (2,1) circle (0.07cm);
\draw[fill,black] (0,7) circle (0.07cm);
\draw[fill,black] (3,0) circle (0.07cm);
\end{tikzpicture}
\end{floatrow}
~\newline
\begin{floatrow}
\begin{tikzpicture}[scale=0.65]
\draw[rotate around={-20:(0.1,6.5)},fill,gray,opacity=0.3] (0.1,6.5) ellipse
(0.15cm and 0.2cm);
\draw[rotate around={-30:(2.1,6.73)},fill,gray,opacity=0.3] (2.1,6.73) ellipse
(0.12cm and 0.1cm);
\draw[rotate around={25:(2.2,7.03)},fill,gray,opacity=0.3] (2.2,7.03) ellipse
(0.12cm and 0.1cm);
\draw[fill,gray,opacity=0.3] (-0.3,4.05) .. controls (-0.2,4.05) and (0.2,4.1) ..
(0.3,4.3) .. controls (0.4,4.5) and (1.2,5.0) .. (1.2,5.2)
.. controls (1.2,5.4) and (0.3,5.5) .. (0.1,5.6) .. controls (-0.1,5.7) and
(-0.2,4.7) .. (-0.2,4.5) -- cycle;
\draw[rotate around={80:(3.1,0.0)},fill,gray,opacity=0.3] (3.1,0.0) ellipse
(0.25cm and 0.2cm);
\draw[rotate around={-20:(1.9,1.1)},fill,gray,opacity=0.3] (1.9,1.1) ellipse
(0.35cm and 0.37cm);
\draw[rotate around={81:(2.27,0.5)},fill,gray,opacity=0.3] (2.27,0.5) ellipse
(0.12cm and 0.1cm);
\draw[rotate around={-10:(5,2.4)},fill,gray,opacity=0.3] (5,2.4) ellipse
(0.15cm and 0.5cm);
\draw[rotate around={51:(5.1,1.7)},fill,gray,opacity=0.3] (5.1,1.7) ellipse
(0.15cm and 0.1cm);
\draw[rotate around={-10:(7.2,1.35)},fill,gray,opacity=0.3] (7.2,1.35) ellipse
(0.15cm and 0.14cm);
\draw[rotate around={-70:(6.8,4.87)},fill,gray,opacity=0.3] (6.8,4.87) ellipse
(0.1cm and 0.2cm);
\draw[rotate around={-20:(7.3,7.2)},fill,gray,opacity=0.3] (7.2,7.1) ellipse
(0.22cm and 0.25cm);
\draw[rotate around={-20:(6.05,6.1)},fill,gray,opacity=0.3] (6.05,6.1) ellipse
(0.34cm and 0.3cm);
\draw[rotate around={-30:(5.7,5.8)},fill,gray,opacity=0.3] (5.7,5.8) ellipse
(0.12cm and 0.1cm);
\draw[rotate around={33:(3.9,0.9)},fill,gray,opacity=0.3] (3.9,0.9) ellipse
(0.15cm and 0.2cm);
\draw[rotate around={87:(3.9,6.1)},fill,gray,opacity=0.3] (3.9,6.1) ellipse
(0.14cm and 0.25cm);
\draw[fill,green,opacity=0.2] (-0.5,7) -- (0,7) -- (1,7) -- (1.5,7.5)
-- (5.5,7.5) -- (6,7) -- (6.5,7.5) -- (7.5,7.5)
-- (7.5,6.5) -- (6.5,6.5) -- (6,6) -- (5.5,6.5) -- (1.5,6.5) -- (1,6)
-- (0,6) -- (-0.5,6);
\draw[fill,blue,opacity=0.2] (-0.5,6) -- (0,6) -- (3,6) -- (3.5,6.5)
-- (5.5,6.5) -- (6,6) -- (6.5,6.5) -- (7,6) -- (7.5,6)
--(7.5,5) -- (7,5) -- (6.5,5.5) -- (6,5) -- (5.5,5.5) -- (3.5,5.5) -- (3,5)
-- (0,5) -- (-0.5,5);
\draw[fill,red,opacity=0.17] (-0.5,5) -- (0,5) -- (3,5) -- (3.5,5.5)
-- (5.5,5.5) -- (6,5) -- (7,5) -- (7.5,5)
-- (7.5,4) -- (7,4) -- (6,4) -- (5.5,4.5) -- (3.5,4.5) -- (3,4) -- (0,4)
-- (-0.5,4);
\draw[fill,violet,opacity=0.2] (-0.5,1) -- (7,1) -- (7.5,1)
-- (7.5,2) -- (7,2) -- (-0.5,2);
\draw[fill,yellow,opacity=0.27] (0.5,7.5) -- (0.5,7) -- (2.5,5) -- (2.5,1)
-- (3.5,0) -- (3.5,-0.5)
-- (2.5,-0.5) -- (2.5,0) -- (1.5,1) -- (1.5,5) -- (-0.5,7) -- (-0.5,7.5);
\draw[fill,orange,opacity=0.24] (5.5,7.5) -- (5.5,3) -- (5.5,2) -- (4.5,1) --
(4.5,-0.5)
-- (3.5,-0.5) -- (3.5,1) -- (4.5,2) -- (4.5,3) -- (4.5,7.5);
\draw[blue!80!black,very thick,opacity=0.4] (-0.5,6) -- (0,6) -- (3,6) -- (3.5,6.5)
-- (5.5,6.5) -- (6,6) -- (6.5,6.5) -- (7,6) -- (7.5,6);
\draw[blue!80!black,very thick,opacity=0.4] (7.5,5) -- (7,5) -- (6.5,5.5)
-- (6,5) -- (5.5,5.5) -- (3.5,5.5) -- (3,5) -- (0,5) -- (-0.5,5);
\draw[red!75!black,very thick,opacity=0.4] (-0.5,5) -- (0,5) -- (3,5) -- (3.5,5.5)
-- (5.5,5.5) -- (6,5) -- (7,5) -- (7.5,5);
\draw[red!75!black,very thick,opacity=0.4]
(7.5,4) -- (7,4) -- (6,4) -- (5.5,4.5) -- (3.5,4.5) -- (3,4) -- (0,4)
-- (-0.5,4);
\draw[green!75!blue,very thick,opacity=0.5]
(-0.5,7) -- (0,7) -- (1,7) -- (1.5,7.5) -- (5.5,7.5) -- (6,7) -- (6.5,7.5)
-- (7.5,7.5);
\draw[green!75!blue,very thick,opacity=0.5]
(7.5,6.5) -- (6.5,6.5) -- (6,6) -- (5.5,6.5) -- (1.5,6.5) -- (1,6)
-- (0,6) -- (-0.5,6);
\draw[violet,very thick,opacity=0.4] (-0.5,1) -- (7,1) -- (7.5,1);
\draw[violet,very thick,opacity=0.4] (7.5,2) -- (7,2) -- (-0.5,2);
\draw[yellow!75!red,very thick,opacity=0.5] (0.5,7.5) -- (0.5,7) -- (2.5,5)
-- (2.5,1) -- (3.5,0) -- (3.5,-0.5);
\draw[yellow!75!red,very thick,opacity=0.5]
(2.5,-0.5) -- (2.5,0) -- (1.5,1) -- (1.5,5) -- (-0.5,7) -- (-0.5,7.5);
\draw[orange,very thick,opacity=0.5] (5.5,7.5) -- (5.5,3) -- (5.5,2) -- (4.5,1) --
(4.5,-0.5);
\draw[orange,very thick,opacity=0.5]
(3.5,-0.5) -- (3.5,1) -- (4.5,2) -- (4.5,3) -- (4.5,7.5);
\draw (-0.5,-0.5) -- (-0.5,7.5) -- (7.5,7.5) -- (7.5,-0.5) -- (-0.5,-0.5);
\draw[white] (-1,-1) -- (-1,8) -- (8,8) -- (8.0,-1.0) -- (-1.0,-1.0);
\end{tikzpicture}
\end{floatrow}
~\newline
\begin{floatrow}
\begin{tikzpicture}[scale=0.65]
\draw[fill,green,opacity=0.2] (-0.5,7) -- (0,7) -- (1,7) -- (1.5,7.5)
-- (5.5,7.5) -- (6,7) -- (6.5,7.5) -- (7.5,7.5)
-- (7.5,6.5) -- (6.5,6.5) -- (6,6) -- (5.5,6.5) -- (1.5,6.5) -- (1,6)
-- (0,6) -- (-0.5,6);
\draw[fill,blue,opacity=0.2] (-0.5,6) -- (0,6) -- (3,6) -- (3.5,6.5)
-- (5.5,6.5) -- (6,6) -- (6.5,6.5) -- (7,6) -- (7.5,6)
--(7.5,5) -- (7,5) -- (6.5,5.5) -- (6,5) -- (5.5,5.5) -- (3.5,5.5) -- (3,5)
-- (0,5) -- (-0.5,5);
\draw[fill,red,opacity=0.17] (-0.5,5) -- (0,5) -- (3,5) -- (3.5,5.5)
-- (5.5,5.5) -- (6,5) -- (7,5) -- (7.5,5)
-- (7.5,4) -- (7,4) -- (6,4) -- (5.5,4.5) -- (3.5,4.5) -- (3,4) -- (0,4)
-- (-0.5,4);
\draw[fill,violet,opacity=0.2] (-0.5,1) -- (7,1) -- (7.5,1)
-- (7.5,2) -- (7,2) -- (-0.5,2);
\draw[blue!80!black,very thick,opacity=0.4] (-0.5,6) -- (0,6) -- (3,6) -- (3.5,6.5)
-- (5.5,6.5) -- (6,6) -- (6.5,6.5) -- (7,6) -- (7.5,6);
\draw[blue!80!black,very thick,opacity=0.4] (7.5,5) -- (7,5) -- (6.5,5.5)
-- (6,5) -- (5.5,5.5) -- (3.5,5.5) -- (3,5) -- (0,5) -- (-0.5,5);
\draw[red!75!black,very thick,opacity=0.4] (-0.5,5) -- (0,5) -- (3,5) -- (3.5,5.5)
-- (5.5,5.5) -- (6,5) -- (7,5) -- (7.5,5);
\draw[red!75!black,very thick,opacity=0.4]
(7.5,4) -- (7,4) -- (6,4) -- (5.5,4.5) -- (3.5,4.5) -- (3,4) -- (0,4)
-- (-0.5,4);
\draw[green!75!blue,very thick,opacity=0.5]
(-0.5,7) -- (0,7) -- (1,7) -- (1.5,7.5) -- (5.5,7.5) -- (6,7) -- (6.5,7.5)
-- (7.5,7.5);
\draw[green!75!blue,very thick,opacity=0.5]
(7.5,6.5) -- (6.5,6.5) -- (6,6) -- (5.5,6.5) -- (1.5,6.5) -- (1,6)
-- (0,6) -- (-0.5,6);
\draw[violet,very thick,opacity=0.4] (-0.5,1) -- (7,1) -- (7.5,1);
\draw[violet,very thick,opacity=0.4] (7.5,2) -- (7,2) -- (-0.5,2);
\draw (-0.5,-0.5) -- (-0.5,7.5) -- (7.5,7.5) -- (7.5,-0.5) -- (-0.5,-0.5);
\draw[white] (-1,-1) -- (-1,8) -- (8,8) -- (8.0,-1.0) -- (-1.0,-1.0);
\foreach \x/\y in {-0.2/5,1.08/5,2.37/5,3.6/5.5,4.83/5.5,6.1/5,7.25/5}
  \draw[red,thick,->,opacity=0.5] (\x,\y) -- (\x,\y+0.35);
\foreach \x/\y in {-0.2/4,1.08/4,2.37/4,3.6/4.5,4.83/4.5,6.1/4,7.25/4}
  \draw[red,thick,->,opacity=0.5] (\x,\y) -- (\x,\y-0.35);
\foreach \x/\y in {-0.2/6,1.08/6,2.37/6,3.6/6.5,4.83/6.5,6.1/6,7.25/6}
  \draw[blue,thick,->,opacity=0.5] (\x,\y) -- (\x,\y+0.35);
\foreach \x/\y in {-0.2/5,1.08/5,2.37/5,3.6/5.5,4.83/5.5,6.1/5,7.25/5}
  \draw[blue,thick,->,opacity=0.5] (\x,\y) -- (\x,\y-0.35);
\foreach \x/\y in {-0.2/7,1.08/7.08,2.37/7.5,3.6/7.5,4.83/7.5,6.1/7.1,7.25/7.5}
  \draw[green!70!blue,thick,->,opacity=0.5] (\x,\y) -- (\x,\y+0.35);
\foreach \x/\y in {-0.2/6,1.08/6.08,2.37/6.5,3.6/6.5,4.83/6.5,6.1/6.1,7.25/6.5}
  \draw[green!70!blue,thick,->,opacity=0.5] (\x,\y) -- (\x,\y-0.35);
\foreach \x/\y in {-0.2/2,1.03/2,2.27/2,3.5/2,4.73/2,5.97/2,7.2/2}
  \draw[violet,thick,->,opacity=0.5] (\x,\y) -- (\x,\y+0.35);
\foreach \x/\y in {-0.2/1,1.03/1,2.27/1,3.5/1,4.73/1,5.97/1,7.2/1}
  \draw[violet,thick,->,opacity=0.5] (\x,\y) -- (\x,\y-0.35);
\end{tikzpicture}
\begin{tikzpicture}[scale=0.65]
\draw[fill,yellow,opacity=0.27] (0.5,7.5) -- (0.5,7) -- (2.5,5) -- (2.5,1)
-- (3.5,0) -- (3.5,-0.5)
-- (2.5,-0.5) -- (2.5,0) -- (1.5,1) -- (1.5,5) -- (-0.5,7) -- (-0.5,7.5);
\draw[fill,orange,opacity=0.24] (5.5,7.5) -- (5.5,3) -- (5.5,2) -- (4.5,1) --
(4.5,-0.5)
-- (3.5,-0.5) -- (3.5,1) -- (4.5,2) -- (4.5,3) -- (4.5,7.5);
\draw[yellow!75!red,very thick,opacity=0.5] (0.5,7.5) -- (0.5,7) -- (2.5,5)
-- (2.5,1) -- (3.5,0) -- (3.5,-0.5);
\draw[yellow!75!red,very thick,opacity=0.5]
(2.5,-0.5) -- (2.5,0) -- (1.5,1) -- (1.5,5) -- (-0.5,7) -- (-0.5,7.5);
\draw[orange,very thick,opacity=0.5] (5.5,7.5) -- (5.5,3) -- (5.5,2) -- (4.5,1) --
(4.5,-0.5);
\draw[orange,very thick,opacity=0.5]
(3.5,-0.5) -- (3.5,1) -- (4.5,2) -- (4.5,3) -- (4.5,7.5);
\draw (-0.5,-0.5) -- (-0.5,7.5) -- (7.5,7.5) -- (7.5,-0.5) -- (-0.5,-0.5);
\draw[white] (-1,-1) -- (-1,8) -- (8,8) -- (8.0,-1.0) -- (-1.0,-1.0);
\foreach \x/\y in {3.5/-0.2,2.5/1.03,2.5/2.27,2.5/3.5,2.5/4.73,1.51/5.97,0.5/7.2}
  \draw[yellow!75!red,thick,->,opacity=0.5] (\x,\y) -- (\x+0.35,\y);
\foreach \x/\y in {2.5/-0.2,1.5/1.03,1.5/2.27,1.5/3.5,1.5/4.73,0.51/5.97,-0.5/7.2}
  \draw[yellow!75!red,thick,->,opacity=0.5] (\x,\y) -- (\x-0.35,\y);
\foreach \x/\y in {3.5/-0.2,3.53/1.03,4.5/2.27,4.5/3.5,4.5/4.73,4.5/5.97,4.5/7.2}
  \draw[orange,thick,->,opacity=0.5] (\x,\y) -- (\x-0.35,\y);
\foreach \x/\y in {4.5/-0.2,4.53/1.03,5.5/2.27,5.5/3.5,5.5/4.73,5.5/5.97,5.5/7.2}
  \draw[orange,thick,->,opacity=0.5] (\x,\y) -- (\x+0.35,\y);
\end{tikzpicture}
\end{floatrow}
\caption{The above pictures illustrate the proofs of Theorem~\ref{CoveringOfCompactSet} and Proposition \ref{Prop.Bi.Lip.simpl}. For an explanation of the pictures, see the accompanying text.\label{SketchStretching}}
\end{figure}

In the accompanying figure, the strategy of the proof of the proposition is
illustrated.

\textbf{Explanation of Figure \ref{SketchStretching}.}
In the two pictures at the top, we provide an illustration of the proof of Theorem~\ref{CoveringOfCompactSet} due to Alberti, Cs\"ornyei, and Preiss, which reduces the covering assertion to the combinatorial result in Lemma \ref{CombinatorialResult}.
The unit square is divided into $2^{2l}$ squares of size $1/2^l$ (the division is sketched by gray lines). The compact set $K$ corresponds to the gray regions. The black points in these pictures represent the centers of all grid squares which have nonempty intersection with $K$.  The blue and red lines in the second picture at the top represent the graphs of the $1$-Lipschitz maps $f_i$ and $g_j$ which cover the black points, as provided by Lemma \ref{CombinatorialResult}.\smallskip

In the picture at the center and the two pictures at the bottom, we illustrate the construction of the stretching map of Proposition~\ref{Prop.Bi.Lip.simpl} (using a different example). In these pictures, the covering of the set $K$ by strips -- as provided by Lemma \ref{IntersectionFreeCovering} -- is indicated by the framed colored regions, each strip being assigned a separate color. In the picture in the center, the compact set $K$ is also depicted: It corresponds to the gray regions in the background of the strips. The pictures at the bottom illustrate the construction of our stretching map $\phi=(\phi_1,\phi_2)$ in Proposition \ref{Prop.Bi.Lip.simpl}: To define the second component $\phi_2$ of our map $\phi$, we stretch all the horizontal strips by a factor of $1+2\tau$ (see the penultimate picture); similarly, to define the first component $\phi_1$ we stretch all the vertical strips by a factor of $1+2\tau$ (see the last picture).

\begin{proof}[Proof of Proposition \ref{Prop.Bi.Lip.simpl}]
First note that for $|K|=0$ the identity map $\phi=\operatorname{id}$
trivially has all the properties stated in the proposition. Hence we can assume that
$|K|>0$.

\textit{Step 1 (preliminaries).}
 Applying Lemma \ref{IntersectionFreeCovering} to $K$ and $\epsilon=|K|^{1/2},$  there exist
 $\delta=1/2^l>0$ (for some $l$ large enough), $N,M\in \mathbb{N}$ and $1$-Lipschitz functions
 $f_i,g_j:[0,1]\rightarrow [0,1],$ $1\leq i\leq N,$ $1\leq j\leq M$ such that
\begin{equation}
\label{control.N.M}
\delta N\leq 2\sqrt{|K|}\quad \text{and}\quad \delta M\leq 2\sqrt{|K|},
\end{equation}
\begin{equation}
\label{strips.covering.K}
K\subset \big(\bigcup_{i=1}^N\{|f_i(x)-y|\leq \delta\}\big)
\cup \big(\bigcup_{j=1}^M \{|g_j(y)-x|\leq \delta\}\big)\subset [0,1]^2,
\end{equation}
\begin{equation}
\label{control.intersections}
\sum_{i=1}^N\chi_{\{|f_i(x)-y|< \delta\}}\leq 1\quad
\text{and}\quad \sum_{j=1}^M\chi_{\{|g_j(y)-x|< \delta\}}\leq 1,
\end{equation}
where $\{|f_i(x)-y|\leq \delta\}$ is a short notation for
$\{(x,y):x\in [0,1],|f_i(x)-y|\leq \delta\}$ and similarly for the $g_j$.

\textit{Step 2 (definition of $\phi$).}
Define our map $\phi=(\phi_1,\phi_2)$ by
\begin{align*}\phi_2(x,y)&=(1-4\delta N\tau)y+2\tau\sum_{i=1}^{N}
\Big(\min\left(2\delta,y-f_i(x)+\delta\right)-\min\left(0,y-f_i(x)+\delta\right)
\Big)
\\
&=(1-4\delta N\tau)y+2\tau\sum_{i=1}^{N}\left\{\begin{array}{cl}
2\delta& \text{if $y\geq f_i(x)+\delta$}\\
y- f_i(x)+\delta & \text{if $f_i(x)-\delta\leq y\leq f_i(x)+\delta$}\\
0 &\text{if $y\leq f_i(x)-\delta$}
\end{array}
\right.
\end{align*}
and, symmetrically,
\begin{align*}
\phi_1(x,y)&=(1-4\delta M\tau)x+2\tau\sum_{j=1}^{M}
\Big(
\min\left(2\delta,x-g_j(y)+\delta\right)-\min\left(0,x-g_j(y)+\delta\right)
\Big)
\\
&=(1-4\delta M\tau)x+2\tau\sum_{j=1}^{M}\left\{\begin{array}{cl}
2\delta& \text{if $x\geq g_j(y)+\delta$}\\
x- g_j(y)+\delta & \text{if $g_j(y)-\delta\leq x\leq g_j(y)+\delta$}\\
0 &\text{if $x\leq g_j(y)-\delta$}.
\end{array}
\right.
\end{align*}

\textit{Step 3 (properties of $\phi_2$).} First from \eqref{control.N.M} we get that, for every $(x,y)\in [0,1]^2$,
\begin{equation}
\label{phi.2.L.infty}
|\phi_2(x,y)-y|\leq 4\tau \delta N y+4\tau \delta N\leq 16\tau\sqrt{|K|}.
\end{equation}
Together with the corresponding estimate for $\phi_1$, this shows \eqref{prop.L.infty}.
Moreover, since all the strips are contained in $[0,1]^2$
(see the second inclusion in (\ref{strips.covering.K})) we get that
for $1\leq i\leq N$ and $x\in [0,1]$
\begin{align*}
0\leq f_i(x)-\delta\leq f_i(x)+\delta\leq 1
\end{align*}
which directly implies
\begin{equation}\label{phi.2.boundary.horiz}
\phi_2(x,0)=0\quad \text{and}\quad \phi_2(x,1)=1\quad \text{for all $x\in [0,1].$}
\end{equation}
We now turn to the properties of the derivatives of $\phi_2.$
It is elementary to see that $\phi_2$ is Lipschitz and that for a.\,e.\ $(x,y)\in [0,1]^2$ we have
\begin{align*}
\partial_{x}\phi_2(x,y)&=-2\tau \sum_{i=1}^Nf_i'(x)\chi_{\{|f_i(x)-y|<
\delta\}},
\\
\partial_{y}\phi_2(x,y)&=1-4\delta N\tau+2\tau \sum_{i=1}^N\chi_{\{|f_i(x)-y|<
\delta\}}.
\end{align*}
We first deal with $\partial_x\phi_2.$
From the fact that $|f_i'|\leq 1$ and (\ref{control.intersections}),
we immediately obtain
\begin{align*}
\partial_x\phi_2&=0 \quad \text{a.\,e.\ outside
$\bigcup_{i=1}^N\{|f_i(x)-y|< \delta\}$},\\
|\partial_x\phi_2|&\leq 2\tau \quad \text{a.\,e.\ in $\bigcup_{i=1}^N\{|f_i(x)-y|<
\delta\}$}.
\end{align*}
Since (\ref{control.N.M}) yields
\begin{align*}
|\cup_{i=1}^N\{|f_i(x)-y|< \delta\}|\leq 2\delta N\leq 4\sqrt{|K|},
\end{align*}
we directly obtain
\begin{align*}
\int_{[0,1]^2}|\partial_x\phi_2|^p \,dx
=\int_{\cup_{i=1}^N\{|f_i(x)-y|<\delta\}} |\partial_x\phi_2|^p \,dx
\leq 4\sqrt{|K|}(2\tau)^p
\end{align*}
holds, which implies for every $1\leq p\leq \infty$
\begin{equation}\label{phi.2.x.L.p}
\|\partial_x\phi_2\|_{L^p([0,1]^2)}\leq 8\tau |K|^{1/(2p)}.
\end{equation}
We now deal with  $\partial_y\phi_2.$ We first notice that we have
\begin{align}\label{partial.y.phi.2.outside}
&\partial_{y}\phi_2(x,y)=1-4\tau\delta N\quad \text{for }(x,y)\in [0,1]^2\setminus \cup_{i=1}^N\{(x,y)\in [0,1]^2:|f_i(x)-y|\leq \delta\}.
\end{align}
and
\begin{align}\label{partial.y.phi.2.inside}
&\partial_{y}\phi_2(x,y)=1-4\tau\delta N+2\tau \quad \text{for }(x,y)\in \cup_{i=1}^N\{(x,y)\in [0,1]^2:|f_i(x)-y|\leq \delta\}.
\end{align}
This directly yields the lower bound \eqref{prop.derivees.phi.2}. The bound \eqref{prop.derivees.phi.1} is obtained by analogous estimates for $\phi_1$.
Using  (\ref{partial.y.phi.2.outside}) and (\ref{partial.y.phi.2.inside})
and noticing as before that
\begin{align*}
|\{(x,y)\in [0,1]^2: |f_i(x)-y|\leq \delta \text{ for some }i\}|
\leq 2\delta N\leq 4\sqrt{|K|},
\end{align*}
we get for every fixed $x\in [0,1]$
\begin{align*}
&\int_{[0,1]^2}|\partial_{y}\phi_2(x,y)-1|^p dx
\\&
= \int_{\{(x,y)\in [0,1]^2: |f_i(x)-y|>\delta \text{ }\forall i\}}
|\partial_{y}\phi_2(x,y)-1|^p \,dx
\\&~~~
+\int_{\{(x,y)\in [0,1]^2: |f_i(x)-y|\leq \delta \text{ for some }i\}}
|\partial_{y}\phi_2(x,y)-1|^p \,dx
\\&
\leq  (C\tau)^{p}\sqrt{|K|}.
\end{align*}
Thus, for every $x\in [0,1]$ and every $1\leq p\leq \infty$ we have
\begin{equation}\label{phi.2.y.L.p.line}
\|\partial_y\phi_2-1\|_{L^p([0,1]^2)}\leq C
\tau |K|^{1/(2p)}.
\end{equation}
In connection with the analogous estimates for $\phi_1$, \eqref{phi.2.x.L.p} and \eqref{phi.2.y.L.p.line} imply \eqref{prop.W.1.p}.

\textit{Step 4.} In this step we prove the three assertions involving the determinant, namely (\ref{prop.nabla.phi.bounded.by.det})-(\ref{prop.det.global}). For a.\,e.\ $(x,y)\in [0,1]^2$ we have
\begin{align*}
\nabla\phi(x,y)=
\begin{pmatrix}
1-4\tau\delta M+2\tau \sum_{j=1}^M\chi_{\{|g_j(y)-x|< \delta\}}
&
-2\tau \sum_{j=1}^Mg_j'(y)\chi_{\{|g_j(y)-x|<\delta\}}
\\
-2\tau \sum_{i=1}^Nf_i'(x)\chi_{\{|f_i(x)-y|<\delta\}}
&
1-4\tau\delta N+2\tau \sum_{i=1}^N\chi_{\{|f_i(x)-y|< \delta\}}
\end{pmatrix}.
\end{align*}
Denote by $S$ the union of the interior of all the strips, i.e.
\begin{align*}
S:=\big(\cup_{i=1}^N\{|f_i(x)-y|< \delta\}\big)\cup \big(\cup_{j=1}^M
\{|g_j(y)-x|< \delta\}\big).
\end{align*}
Then, obviously, a.\,e.\ outside $S$ we have
\begin{align*}
\nabla\phi=
\begin{pmatrix}
1-4\delta M\tau
& 0
\\
0&
1-4\delta N\tau
\end{pmatrix}.
\end{align*}
Thus, recalling that $M\delta,N\delta\leq 2\sqrt{|K|}$ (see
(\ref{control.N.M})),
\begin{equation}
\label{phi.det.outside.A}
\det\nabla \phi= 1-4\tau(\delta M+\delta N)+16\delta^2MN\tau^2
\geq 1-16\tau \sqrt{|K|}\quad \text{a.\,e.\ in $[0,1]^2\setminus S.$}
\end{equation}
Furthermore, assuming $\tau |K|^{1/2}\leq 1/32$, we have
\begin{equation}
\label{nabla.bounded.by.det.outside.A}
|\nabla \phi|\leq 2\det\nabla\phi\quad \text{a.\,e.\ in $[0,1]^2\setminus S.$}
\end{equation}
Trivially   for all $(x,y)\in S$ we have
\begin{align*}
1\leq \sum_{i=1}^N\chi_{\{|f_i(x)-y|<
\delta\}}+\sum_{j=1}^M\chi_{\{|g_j(y)-x|< \delta\}}.
\end{align*}
Thus we get for a.\,e.\ $(x,y)\in S$, recalling that $|f'_i|,|g'_j|\leq 1$ and
using (\ref{control.intersections}),
\begin{align*}
\det\nabla \phi(x,y)&=\left[1-4\tau\delta M+2\tau
\sum_{j=1}^M\chi_{\{|g_j(y)-x|< \delta\}}\right]\left[1-4\tau\delta N+2\tau
\sum_{i=1}^N\chi_{\{|f_i(x)-y|< \delta\}}\right]\\
&~~~-4\tau^2\left[\sum_{j=1}^Mg_j'(y)\chi_{\{|g_j(y)-x|<
\delta\}}\right]\left[\sum_{i=1}^Nf_i'(x)\chi_{\{|f_i(x)-y|<
\delta\}}\right]
\\
&= 1+2\tau\left[
\left(\sum_{j=1}^M\chi_{\{|g_j(y)-x|< \delta\}}
+\sum_{i=1}^N\chi_{\{|f_i(x)-y|< \delta\}}\right)-2\delta (M+N)\right]\\
&~~~
+4\tau^2
\Bigg(
\left[\sum_{j=1}^M \chi_{\{|g_j(y)-x|<
\delta\}}\right]\left[\sum_{i=1}^N \chi_{\{|f_i(x)-y|<
\delta\}}\right]
\\
&~~~~~~~~~~~~~~~
-\left[\sum_{j=1}^Mg_j'(y)\chi_{\{|g_j(y)-x|\leq
\delta\}}\right]\left[\sum_{i=1}^Nf_i'(x)\chi_{\{|f_i(x)-y|<
\delta\}}\right]\\
&~~~~~~~~~~~~~~~+4\delta M\delta N-2 \delta N
\sum_{j=1}^M\chi_{\{|g_j(y)-x|< \delta\}}
-2 \delta M
\sum_{i=1}^N\chi_{\{|f_i(x)-y|< \delta\}}
\Bigg)
\\
&\geq 1+2\tau\left[1-2\delta (M+N)\right]-8\tau^2\delta(M+N)
\\
&= 1+2\tau[1-2\delta (M+N)-8\tau\delta(M+N)].
\end{align*}
Recalling (\ref{control.N.M}), we may choose the upper bound $c$ for $|K|$ and $\sqrt{|K|}\tau$ in the assumptions of the proposition small enough so that
\begin{align*}
1+2\tau[1-2\delta (M+N)-8\tau\delta(M+N)]
\geq 1+\tau.
\end{align*}
Combining the last two estimates  gives
\begin{equation}
\label{phi.det.A}
\det\nabla \phi\geq 1+\tau \quad \text{a.\,e.\ in $S.$}
\end{equation}
Since $K\subset S$ (see (\ref{strips.covering.K})), the last inequality
obviously holds true a.\,e.\ in $K.$  Combining (\ref{phi.det.outside.A}) and the
last estimate gives
\begin{align*}
\det \nabla \phi\geq 1-16\sqrt{|K|}\tau\quad
\text{a.\,e.\ on $[0,1]^2.$}
\end{align*}
Taking into account the bound $|\nabla \phi-\Id|\leq C\tau$, we also obtain
\begin{align*}
|\nabla \phi| \leq \det\nabla \phi \quad \text{a.\,e.\ in }S.
\end{align*}

\textit{Step 5.}  We finally prove that $\phi$ is bi-Lipschitz from $[0,1]^2$ to $[0,1]^2.$
Taking $\tau\sqrt{|K|}\leq 1/32$, the last inequality implies
\begin{align*}
\det\nabla \phi\geq 1/2>0\quad \text{a.\,e.\ in $\Omega.$}
\end{align*}
By (\ref{phi.2.boundary.horiz}), \eqref{prop.W.1.p} for $p=\infty$, \eqref{prop.derivees.phi.2}, and the analogous estimates for $\phi_1$, we deduce that
\begin{align*}\phi|_{\partial (0,1)^2}:\partial (0,1)^2\rightarrow \partial (0,1)^2\quad\text{is bi-Lipschitz.}
\end{align*}
Hence, classical results based on the degree theory show that $\phi:[0,1]^2\rightarrow [0,1]^2$ is also bi-Lipschitz (see e.\,g.\ Theorem 2 in the work of Ball \cite{Ball2}; note that $\phi|_{\partial (0,1)^2}$ can be easily extended to a homeomorphism from $[0,1]^2$ to $[0,1]^2$).
\end{proof}
\begin{appendix}
\section{Decomposition of $C^{1,1}$-domains into subdomains that are $C^{1,1}$-equivalent to the unit square}

The following lemma has been used in Section \ref{sub.sect.Simplification.of.the.domain}.

\begin{lemma}\label{lemma:decomposition.domain}
Let $\Omega\subset \mathbb{R}^2$ be a bounded domain with boundary of class $C^{1,1}$. Then there exists a finite number of open sets $A_i\subset \Omega$ such that:
\begin{itemize}
\item For every $A_i$ there exists a $C^{1,1}$-diffeomorphism which maps $\overline{A_i}$ to the unit square $[0,1]^2$.
    \item The sets $A_i$ are pairwise disjoint.
\item We have $\overline{\Omega}=\bigcup_i \overline{A_i}$.
\end{itemize}
\end{lemma}
\begin{proof}
We first summarize the proof.
\begin{itemize}
\item First we reduce the problem to the case of a polygon: We approximate $\Omega$ from inside by a polygon whose boundary follows closely the boundary $\partial\Omega$ (in a sense to be made explicit below). The outer part may then be decomposed as shown in Figure \ref{ApproximateBoundaryByPolygon}. The outer quadrilaterals (which have one curved side) are easily mapped by a $C^{1,1}$-diffeomorphism to the unit square. It then just remains to treat the case of the (inner) polygon (which possibly contains holes).
\item Any polygon (including polygons with holes) may be decomposed into a finite number of triangles, so the problem is reduced to the case of a triangle. Since every triangle may be mapped by an affine map to the reference triangle with corners $(0,0)$, $(1,0)$, and $(0,1)$, we only need to consider the case of the reference triangle.
\item The reference triangle may be covered by three convex quadrilaterals as shown in Figure \ref{CoverTriangle}.
Applying Lemma \ref{lemma:bilinear}, each of these quadrilaterals can be mapped by a (bi-linear) $C^{\infty}$-diffeomorphism to the unit square.
\end{itemize}
It only remains to make explicit the first point. We split its proof into three steps.\smallskip

\textit{Step 1.} We can find (see e.\,g.\ Theorem 1.2.6 in \cite{KrantzParks}) a function $F:\mathbb{R}^2 \rightarrow \mathbb{R}$ with the following properties:
\begin{itemize}
\item $F$ is of class $C^{1,1}$,
\item we have $\{F=0\}=\partial\Omega$ and $\{F>0\}=\Omega$,
\item $\nabla F$ is nonzero on $\partial\Omega$.
\end{itemize}
Take $\delta>0$ small enough and construct a closed polygon $P$ whose boundary is contained in the set $\{\delta/2<F<2\delta\}$, whose vertices lie on the curve $\{F=\delta\}$, and whose edges have length of order $\delta$. Note that for $\delta$ small enough, such a polygon exists. Connect every vertex of $P$ to the boundary of $\partial \Omega$ with a line segment pointing in direction $-\nabla F$ (the green line segments in Figure \ref{ApproximateBoundaryByPolygon}). Let us denote the open regions which are bounded by the polygon $P$ on one side and these line segments and $\partial\Omega$ on the other sides as $B_i$ (see Figure \ref{ApproximateBoundaryByPolygon}). We will show in the remaining two steps that every $\overline{B_i}$ is $C^{1,1}$-diffeomorphic to the unit square $[0,1]^2.$\smallskip

\textit{Step 2.} Fix some $i$ and let $z_1,z_2,w_1,w_2$ be the four "vertices" of $B_i$
with $z_1,z_2\in \{F=\delta\}$ and $w_1,w_2\in \partial \Omega$ ordered such that
$$[z_1,z_2]\cup [z_1,w_1]\cup [z_2,w_2]$$ is the "non-curved" boundary of $B_i$ (see Figure \ref{ApproximateBoundaryByPolygon}).
Moreover, for $\delta>0$ small enough, the "curved" boundary of $B_i$ is (up to a rotation) the graph of some $C^{1,1}$ function.
Since $F(z_1)-F(z_2)=0$ we have
$$\nabla F(z)\cdot (F(z_1)-F(z_2))=0$$ for some
 $z\in [F(z_1),F(z_2)]$.
Hence since $F\in C^{1,1}$ we deduce that, for $j=1,2$,
$$|\nabla F(z_j)\cdot (z_1-z_2)|=|\nabla (F(z_j)-\nabla F(z))\cdot (z_1-z_2)|\leq C(F)|z_1-z_2|^2.$$
Since $|z_1-z_2|$ is of order $\delta$ and $|\nabla F(z_j)|$ is of order $1$ (since $\nabla F\neq 0$ on $\partial \Omega$) we get that the cosine of the angle between $z_1-z_2$ and $z_1-w_1$ as well as between $z_1-z_2$ and $z_1-w_1$ is of order $\delta$.
This implies in particular that the two lines
$$l_1(s):=z_1+s(w_1-z_1)/|w_1-z_1|\quad \text{and}\quad l_2(t):=z_2+t(w_2-z_2)/|w_2-z_2|$$
can only intersect at a point at distance at least of order $1$ from the segment $[z_1,z_2]$.
Recalling that the "curved" boundary of $B_i$ is at distance of order $\delta$ from the segment $[z_1,z_2]$, we can therefore find some $\overline{s}>|w_1-z_1|$ and $\overline{t}>|w_2-z_2|$ such the convex hull of the points
$$\{z_1,z_2,z_3:=l_2(\overline{t}),z_4:=l_1(\overline{s})\},$$ which we denote by $P_i$ (see Figure \ref{ApproximateBoundaryByPolygon}),
is a convex quadrilateral containing $B_i$ and whose boundary consists of
$$[z_1,z_2]\cup [z_2,z_3]\cup[z_3,z_4]\cup[z_4,z_1].
$$

\textit{Step 3.}
Using Lemma \ref{lemma:bilinear} there exists a (bi-linear) $C^{\infty}$- differomorphism $\varphi_i$ mapping $P_i$ to $[0,1]^2.$ Up to a rotation we have that $\varphi_i(\overline{B_i})$ is the region delimited by
$$[(0,0),(1,0)]\cup [(0,0),(0,h_i(0))]\cup [(1,0),(1,h_i(1))]\cup \{(x,h_i(x)):x\in [0,1]\}$$
for some $h_i\in C^{1,1}([0,1],(0,\infty))$.
Finally the map
$$(x,y)\rightarrow (x,h_i(x)y)$$ is trivially seen to be a $C^{1,1}$- diffeomorphism from $[0,1]^2$ to $\varphi_i(\overline{B_i})$  which concludes the proof.

\end{proof}
\begin{figure}
\begin{tikzpicture}[scale=0.6]
\draw[Blue,fill=Blue!15!white] (0.4,0.2) -- (1.2,2.7) -- (5.5,4.5) -- (4.5,1.0) -- (5.5,-1) -- (3,-0.7) -- cycle;
\draw[Green,very thick] (4.5,1.0) -- ($ (4.5,1.0)+0.7/sqrt(1.0^2+3.5^2)*(1.0,3.5)+0.7/sqrt(1.0^2+2.0^2)*(1.0,-2.0)$);
\draw[Green,very thick] (5.5,-1) -- ($(5.5,-1)-0.27/sqrt(2.5^2+0.3^2)*(-2.5,0.3)-0.27/sqrt(1.0^2+2.0^2)*(-1.0,2.0)$);
\draw[Green,very thick] (3,-0.7) -- ($ (3,-0.7)-2.0/sqrt(2.5^2+0.3^2)*(2.5,-0.3)-2.0/sqrt(2.6^2+0.9^2)*(-2.6,0.9)$);
\draw[Red,very thick] ($ (3,-0.7)-2.0/sqrt(2.5^2+0.3^2)*(2.5,-0.3)-2.0/sqrt(2.6^2+0.9^2)*(-2.6,0.9)$) -- ($ (3,-0.7)-13.0/sqrt(2.5^2+0.3^2)*(2.5,-0.3)-13.0/sqrt(2.6^2+0.9^2)*(-2.6,0.9)$);
\draw[Red,very thick] ($(0.4,0.2)-5*0.32/sqrt(0.8^2+2.5^2)*(0.8,2.5)-5*0.32/sqrt(2.6^2+0.9^2)*(2.6,-0.9)$) -- ($ (3,-0.7)-13.0/sqrt(2.5^2+0.3^2)*(2.5,-0.3)-13.0/sqrt(2.6^2+0.9^2)*(-2.6,0.9)$);
\draw[Green,very thick] (0.4,0.2) --($(0.4,0.2)-0.32/sqrt(0.8^2+2.5^2)*(0.8,2.5)-0.32/sqrt(2.6^2+0.9^2)*(2.6,-0.9)$);
\draw[Red,very thick] ($(0.4,0.2)-0.32/sqrt(0.8^2+2.5^2)*(0.8,2.5)-0.32/sqrt(2.6^2+0.9^2)*(2.6,-0.9)$) --($(0.4,0.2)-5*0.32/sqrt(0.8^2+2.5^2)*(0.8,2.5)-5*0.32/sqrt(2.6^2+0.9^2)*(2.6,-0.9)$);
\draw[Green,very thick] (1.2,2.7) -- ($(1.2,2.7)-0.4/sqrt(0.8^2+2.5^2)*(-0.8,-2.5)-0.4/sqrt(4.3^2+1.8^2)*(4.3,1.8)$);
\draw[Green,very thick](5.5,4.5)--($(5.5,4.5)-0.365/sqrt(1.0^2+3.5^2)*(-1.0,-3.5)-0.365/sqrt(4.3^2+1.8^2)*(-4.3,-1.8)$);
\draw[Blue,thick] (0.4,0.2) -- (1.2,2.7) -- (5.5,4.5) -- (4.5,1.0) -- (5.5,-1) -- (3,-0.7) -- cycle;
\draw[thick] (0,0) .. controls (-1,1) and (0.5,2.5) .. (1,3) .. controls (1.5,3.5) and (5,5) .. (6,5) .. controls (7,5) and (5,2) .. (5,1) .. controls (5,0) and (6,0) .. (6,-1) .. controls (6,-2) and (3,-1) .. (2,-1) .. controls (1,-1) and (1,-1) .. (0,0);
\draw (3.0,1.2) node{$P$};
\draw (3,-0.1) node{$z_1$};
\draw (2.4,-1.5) node{$w_1$};
\draw (1,0.5) node{$z_2$};
\draw (-0.7,0.12) node{$w_2$};
\draw (3,-3.4) node{$z_4$};
\draw (-2,-1) node{$z_3$};
\draw (0.2,1.2) node{$B_1$};
\draw (1.2,-0.5) node{$B_i$};
\draw (0.5,-1.5) node{$P_i$};
\end{tikzpicture}
\caption{Reduction to the case of a polygon.\label{ApproximateBoundaryByPolygon}}
\end{figure}
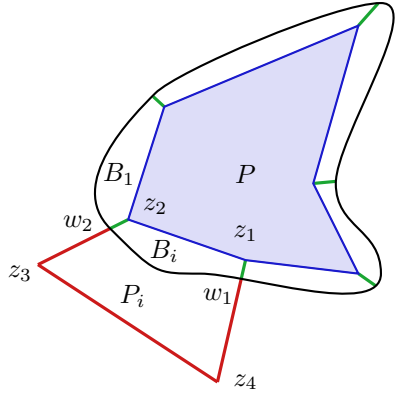

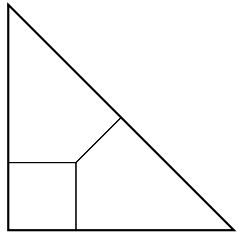
\begin{figure}
\begin{tikzpicture}[scale=0.6]
\draw[thick] (0,0) -- (5,0) -- (0,5) -- cycle;
\draw (0,0) -- (1.5,0) -- (1.5,1.5) -- (0,1.5) -- cycle;
\draw (1.5,0) -- (5,0) -- (2.5,2.5) -- (1.5,1.5) -- cycle;
\draw (0,1.5) -- (0,5) -- (2.5,2.5) -- (1.5,1.5) -- cycle;
\end{tikzpicture}
\caption{Covering of a triangle by three convex quadrilaterals.\label{CoverTriangle}}
\end{figure}
\begin{definition} A map $\varphi:\mathbb{R}^2\rightarrow \mathbb{R}^2$ is called bi-linear if
 for every $x,y\in \mathbb{R}$ the functions $\varphi(x,\cdot)$ and $\varphi(\cdot,y)$ are affine.
 Equivalently a map $\varphi$ is bi-linear if there exist eight constants $a_1,a_2,a_3,a_4,b_1,b_2,b_3,b_4\in \mathbb{R}$ such that
$$\varphi(x,y)=(a_1xy+a_2x+a_3y+a_4,b_1xy+b_2x+b_3y+b_4),\quad \text{for $(x,y)\in \mathbb{R}^2$}.$$
\end{definition}
In the previous proof we used the following elementary lemma.
\begin{lemma}\label{lemma:bilinear} For any closed convex quadrilateral $P\subset \mathbb{R}^2$ there exists a bi-linear map $\varphi\in \operatorname{Diff}^{\infty}([0,1]^2;P)$.
\end{lemma}
\begin{proof} Denote by $z_1,z_2,z_3,z_4$ the vertices of $P$ ordered so that $\partial P=[z_1,z_2]\cup [z_2,z_3]\cup [z_3,z_4]\cup [z_4,z_1].$
With no loss of generality, using an affine map (which preserves the convexity; note that the composition of a bi-linear map with an affine map is still bi-linear), one can assume that
\begin{align*}
z_1=(0,0),\quad z_2=(1,0)\quad \text{and}\quad  z_4=(0,1).
\end{align*}
Let us denote $z_3=(\alpha,\beta).$ Since $P$ is convex we easily deduce that
\begin{equation}\label{condition.convex}\alpha>0,\quad \beta>0\quad \text{and}\quad\alpha+\beta>1.\end{equation}
We claim that the map
\begin{align*}
\varphi(x,y)=((\alpha-1)xy+x,(\beta-1)xy+y)
\end{align*}
has all the wished properties.
First using (\ref{condition.convex}), one easily deduces that $\varphi$ is one-to-one in $[0,1]^2$ and that $\det\nabla \varphi>0$ in $[0,1]^2$. Thus $\varphi\in \operatorname{Diff}^{\infty}([0,1]^2;\varphi([0,1]^2)).$
Finally since
\begin{align*}
\varphi(0,0)=z_1,\quad \varphi(1,0)=z_2,\quad \varphi(1,1)=z_3\quad \text{and}\quad \varphi(0,1)=z_4,
\end{align*}
we have $\varphi(\partial [0,1]^2)=\partial P$ and hence by degree theory $\varphi([0,1]^2)=P$, proving the lemma.
\end{proof}
\end{appendix}

\bibliographystyle{plain}
\bibliography{jacobian_inequality_Lp}

\end{document}